\pdfoutput=1
\documentclass[11pt]{article}
\usepackage[margin=1in]{geometry}
\usepackage[T1]{fontenc}
\usepackage{lmodern,microtype}
\usepackage{amsmath,amssymb,amsthm}
\usepackage{booktabs}
\usepackage{flafter}
\usepackage[section]{placeins}
\usepackage{xcolor}
\usepackage{etoolbox}
\usepackage[colorlinks=true,linkcolor=blue!50!black,citecolor=blue!50!black,urlcolor=blue!50!black,bookmarksnumbered=true]{hyperref}
\usepackage{bookmark}
\hypersetup{pdftitle={Positive-temperature spin-glass order on the three-dimensional Migdal\textendash Kadanoff lattice},pdfauthor={Yan Ru Pei},pdfkeywords={Edwards\textendash Anderson model, hierarchical lattice, Migdal\textendash Kadanoff renormalization, spin-glass order, Rogozin inequality, computer-assisted proof, Chayes\textendash Machta\textendash Redner representation, blue-cluster imbalance}}
\newtheorem{theorem}{Theorem}[section]
\newtheorem{proposition}[theorem]{Proposition}
\newtheorem{lemma}[theorem]{Lemma}
\newtheorem{corollary}[theorem]{Corollary}
\newtheorem{mainthm}{Theorem}

\theoremstyle{definition}
\newtheorem{definition}[theorem]{Definition}
\theoremstyle{remark}
\newtheorem{remark}[theorem]{Remark}
\numberwithin{equation}{section}
\newcommand{\E}{\mathbb E}
\newcommand{\Pp}{\mathbb P}

\newcommand{\Z}{\mathbb Z}
\makeatletter\g@addto@macro\bfseries{\boldmath}\makeatother

\newcommand{\1}{\mathbf 1}
\newcommand{\DL}{\mathbb D}
\newcommand{\V}{\mathcal V}
\newcommand{\Imb}{\mathcal I}
\newcommand{\eff}{\mathrm{eff}}
\newcommand{\artanh}{\operatorname{artanh}}
\newcommand{\sech}{\operatorname{sech}}
\newcommand{\sgn}{\operatorname{sgn}}
\newcommand{\Hyp}{\mathrm{H}}
\newcommand{\Ts}{T^{\mathrm s}}
\newcommand{\Tb}{T^{\mathrm b}}
\newcommand{\as}{a^{\mathrm s}}
\newcommand{\ab}{a^{\mathrm b}}
\newcommand{\bareps}{\bar\varepsilon}
\newcommand{\sH}{\mathsf H}
\newcommand{\hsH}{\widehat{\mathsf H}}
\newcommand{\cE}{\mathcal E}
\newcommand{\minT}{\operatorname{min2}}
\newcommand{\lpk}{\preceq}
\newcommand{\gpk}{\succeq}
\newcommand{\repocommit}{21744230ff7d}
\makeatletter
\patchcmd{\l@section}{\addvspace{1.0em \@plus\p@}}{\addvspace{0.55em \@plus\p@}}{}%
  {\PackageWarningNoLine{mk-order}{contents spacing patch failed}}
\makeatother

\title{Positive-temperature spin-glass order\\
on the three-dimensional Migdal--Kadanoff lattice}
\author{Yan Ru Pei\\[0.3em]\small\href{mailto:yanrpei@gmail.com}{\texttt{yanrpei@gmail.com}}}
\date{September 2026}

\begin{document}
\maketitle

\begin{abstract}
We study the Ising spin glass with independent, identically distributed couplings on the diamond
hierarchical lattice with $n$ branches, of effective dimension $1+\log_2 n$; $n=4$ is the
Migdal--Kadanoff lattice of $\mathbb{Z}^3$. For $n\ge4$ and every coupling law with a bounded
density we prove that at sufficiently low temperature $T>0$, uniformly in the level of the lattice,
the poles are ordered, every spin has Edwards--Anderson parameter $1-O(T)$ when the poles are
fixed, and the overlap of two free-boundary replicas has second moment $1-O(T)$. For standard
Gaussian couplings, $n=3,4,5,7,8$ and every $T\le T^*(n)=4/15,\,6/7,\,6/5,\,7/4,\,2$
($0.89$--$0.98$ times numerical estimates of $T_c$), a computer-assisted proof gives pole order
and, as the level tends to infinity, site-overlap order and the blue-cluster statements below;
likewise for every larger $n$ at these temperatures. Previously, spin-glass order on these lattices
with frustrated independent couplings had been proved only for more than $10^{18}$ branches
(Collet and Eckmann, 1984). The key step is growth of the effective coupling under the exact
renormalization map: for $n\ge4$, by a one-step contraction, via Rogozin's inequality, of a
density bound on an event of high probability; near $T_c$, by interval-arithmetic bounds in the
peakedness order. A transfer theorem following the ancestor-chain argument of Collet and Eckmann
turns this growth into bulk statements. In the Chayes--Machta--Redner representation the same
input gives a blue cluster that contains both poles with probability tending to one and whose
expected density is bounded away from zero, while the largest other blue cluster has density
tending to zero in probability and in $L^1$; hence blue-cluster density imbalance. To our
knowledge this is the first proof of such an imbalance for a spin glass on a graph other than the
complete graph.
\end{abstract}

\smallskip


\tableofcontents

\section{Introduction}\label{sec:intro}

\subsection{Spin glasses on diamond lattices}\label{sec:problem}

Whether the Edwards--Anderson spin glass \cite{EA1975} on $\Z^3$ has an ordered phase at
positive temperature is open. The Migdal--Kadanoff approximation \cite{Migdal1975,Kadanoff1976}
replaces $\Z^d$ by a hierarchical lattice on which the renormalization map is exact
\cite{BerkerOstlund1979,KaufmanGriffiths1981,GriffithsKaufman1982}. With scale factor $2$ (two
pieces in series on each branch) this is the diamond lattice $\DL_N$: $\DL_0$ is one edge between two poles $A$ and $B$, and $\DL_N$
consists of $n$ parallel branches between $A$ and $B$, each made of two copies of $\DL_{N-1}$
joined in series. Its effective dimension is $d_\eff=1+\log_2n$, and $n=2^{d-1}$ is the
Migdal--Kadanoff approximation of $\Z^d$; for $n=4$ it is three-dimensional in this sense.
With independent, identically distributed couplings $J_e\sim J$ at inverse temperature $\beta$,
the effective coupling $\xi_N$ between the poles is defined by requiring that the Boltzmann weight
$\exp(\beta\sum_{e=xy}J_e\sigma_x\sigma_y)$, summed over all spins other than $\sigma_A,\sigma_B$,
be proportional to $e^{\xi_N\sigma_A\sigma_B}$ (so that $\langle\sigma_A\sigma_B\rangle=\tanh\xi_N$
with free poles). It obeys the exact distributional recursion \cite[(3.1)]{CE1984}
\begin{equation}\label{eq:intro-rec}
\xi_{k+1}\overset{d}{=}\sum_{i=1}^{n}\xi_k^{(i,1)}\star\xi_k^{(i,2)},\qquad
a\star b:=\artanh(\tanh a\tanh b),\qquad \xi_0=\beta J,
\end{equation}
with $2n$ independent copies of $\xi_k$ on the right. Spin-glass order at the poles means that
$|\xi_N|\to\infty$ in probability as $N\to\infty$, equivalently
$\E\langle\sigma_A\sigma_B\rangle^2\to1$.

Physicists have analysed \eqref{eq:intro-rec} since the 1970s. Southern and Young \cite{SY1977}
treated the three-dimensional spin glass by this real-space rescaling; domain-wall scaling
arguments \cite{BM1984,McMillan1984} relate the growth of the effective coupling to the
stiffness exponent and the lower critical dimension (the dimension below which there is no
ordered phase at positive temperature); Gardner \cite{Gardner1984} analysed the
limit of many branches; Berker and coworkers determined phase diagrams from the exact
recursion of coupling distributions (see \cite{HB2005} and the references there); and
Nogueira et al.\ \cite{NCNCA1997} studied local magnetizations of the spin glass on diamond
lattices through an exact recursion for single-site magnetizations that Morgado, Coutinho and
Curado \cite{MCC1990} introduced for the ferromagnet. A numerical renormalization-group study of
$\pm J$ couplings on Migdal--Kadanoff lattices with scale factor $3$ places the lower critical
dimension at $2.520$ \cite{DTB2015}; Okuyama and Ohzeki \cite{OO2026} show that, for $\pm J$
couplings and large scale factors, order can be absent at effective dimensions arbitrarily close
to $3$ (their setting does not overlap with ours; see below). Population-dynamics numerics by the
author for \eqref{eq:intro-rec} (iterating the recursion on a large sample drawn from the current
empirical law), run for $n=3,4,5,7,8$, give the following estimates of the critical temperature
$T_c(n)$, at which the typical size of $|\xi_k|$ switches from decay to growth:
$T_c(3)\in[0.28,0.30]$, $T_c(4)\in[0.88,0.89]$, $T_c(5)\in[1.25,1.30]$,
$T_c(7)\in[1.80,1.85]$ and $T_c(8)\in[2.05,2.10]$, in units in which $J_e$ is standard Gaussian.
These numerics are not part of any proof below and are not included in the companion repository.

Rigorous results are few. Collet and Eckmann \cite{CE1984} proved that for large $\beta$,
fixed pole spins and even coupling laws with finite variance and bounded density, every spin has
$\E\langle\sigma_x\rangle^2>1/2$ \cite[Thm.~7.1]{CE1984}, provided the number of branches is even
(a condition they remark can be dropped \cite[p.~385]{CE1984}) and exceeds $n_0$. Their proof
takes $n_0=\max(1600/Q^3,(40\cdot320^3)^2)>10^{18}$ \cite[Thm.~4.1 and p.~386]{CE1984}, where $Q$
depends on the coupling law; they regard $n$ as ``fixed, and sufficiently large throughout''
\cite[p.~384]{CE1984}. Collet, Eckmann, Glaser and Martin \cite{CEGM1984} had treated an
unfrustrated variant in which the couplings of a layer are equal, and, for couplings whose modulus
is almost surely bounded and bounded away from $0$, proved an Edwards--Anderson transition for
every $n\ge2$ \cite[Thms.~4.4 and~5.2]{CEGM1984}; see also
\cite{CEGM1984b} for a simplified recursion. Collet \cite{Collet1985} reviews these models.
Kamphorst \cite{Kamphorst1986}, extending \cite{CE1984} to coupling laws that need not be even,
proved ``the existence of a mixed ferromagnetic (or antiferromagnetic)--spin-glass fixed point for
an Ising spin-glass model on the diamond hierarchical lattice'' (abstract): for $n$ sufficiently
large (no threshold is given), the map induced by \eqref{eq:intro-rec} on coupling densities has a
fixed point, unique in a small neighbourhood of the Gaussian with mean and variance $1/n$, at which
the linearization has two eigenvalues $2+O(n^{-1/2})$ and the rest of its spectrum strictly inside
the unit disk \cite[Thm.~1]{Kamphorst1986}. The fixed point is thus unstable in two directions; the
paper does not treat the flow from given i.i.d.\ couplings or Edwards--Anderson order, and its
phase diagram is derived heuristically \cite[\S2]{Kamphorst1986}. Koukiou \cite{Koukiou1995} proved a
mean-field-type phase transition in the quenched free energy of a Gaussian model on the same
lattices; at $\beta=0$ that model has entropy $\ln2$ per bond, as for independent bond variables
rather than Ising spins at the sites, so it is not the Edwards--Anderson model on $\DL_N$, and the
paper makes no Edwards--Anderson order statement.

In the opposite direction, Okuyama
and Ohzeki \cite{OO2026} proved the absence of spin-glass order for $\pm J$ couplings and even
numbers of branches under a sufficient condition \cite[Thm.~1]{OO2026}, which for scale factor
$2$ holds only for $n=2$; since our theorems assume a density, the two results do not overlap.
They note that rigorous analyses ``have largely been restricted to models with sufficiently
large fractal dimensions'' \cite[p.~1]{OO2026}, that their argument ``does not apply to
continuous coupling distributions such as the Gaussian distribution'', and that ``several
important cases remain open, including the case of an odd branching number, continuous coupling
distributions, and even-branching MK lattices for which the sufficient condition is not
satisfied'' \cite[pp.~4--5]{OO2026}. Here we prove order for coupling laws with a bounded density
at low temperature (Theorem~\ref{thm:A}, $n\ge4$) and for standard Gaussian couplings at every
$T\le T^*(n)$ (Theorem~\ref{thm:B}, including the odd values $n=3,5,7$); $\pm J$ couplings, which
have no density, remain open on these lattices for every $n\ge3$ (Section~\ref{sec:discussion}).

Besides mean-field models such as the Sherrington--Kirkpatrick model \cite{Talagrand2006},
rigorous positive-temperature spin-glass order is known in other settings, for example: on
the Bethe lattice with boundary fields \cite{CCST1986}, for models with local interactions on
expander graphs \cite{Placke2025}, and on the Nishimori line of a Dyson-hierarchical and a
one-dimensional long-range spin glass \cite{OO2026b}. Our results concern a different class:
Migdal--Kadanoff lattices with as few as three branches and frustrated independent couplings with
a density.

\subsection{The obstruction}\label{sec:obstruction}

Collet and Eckmann obtain growth of the effective coupling from growth of its variance, and
they need a density bound to rule out that $\xi_N$ stays at zero with large probability:
``The divergence of the variance as shown in Theorem 4.1.1 does not imply that the random
variable is large (almost surely), since it could take with small probability a very large
value, while with large probability it could take the value zero. The purpose of Theorem 4.1.2
is to exclude this possibility'' \cite[p.~385]{CE1984} (Theorems 4.1.1 and 4.1.2 there are the
two parts of \cite[Thm.~4.1]{CE1984}). Both steps rest on the central limit
theorem in the number of branches (a Berry--Esseen estimate and a local limit theorem), which is
why $n$ must be huge. At small $n$ the sum in \eqref{eq:intro-rec} has three or four terms and
no normal approximation is available; what is needed is anti-concentration of $\xi_k$ at scales
of order one, improving from level to level. Two facts make this delicate. First, the series
map $\star$ can send bounded densities to densities with a logarithmic singularity at $0$
(Remark~\ref{rem:event}), so a density bound cannot be carried through a single branch; ours bounds
the density of the sum over branches, restricted to an event of high probability
(Section~\ref{sec:mechanism}). Second, for $n=3$ no one-step bound on the supremum of the density
can contract at zero temperature: there $a\star b$ becomes $\sgn(ab)\min(|a|,|b|)$, and a uniform
input law yields a density whose supremum is larger by exactly the factor $11/10$
(Remark~\ref{rem:sharp}). By scaling, the same is expected at large finite $\beta$. This limits
one-step bounds, not coupling growth: Theorem~\ref{thm:B} covers $n=3$.

\subsection{Results}\label{sec:results-intro}

Theorem~\ref{thm:A} (analytic) states that for every $n\ge4$ and every law of $J_e$ with a
bounded density there is $\beta_0<\infty$ such that the following hold for every
$\beta\ge\beta_0$. The poles are ordered:
$\E\langle\sigma_A\sigma_B\rangle^2=\E\tanh^2\xi_N\ge1-O(1/\beta)$ for every $N$, and
$\E\langle\sigma_A\sigma_B\rangle^2\to1$ as $N\to\infty$. Uniformly in $N$, with fixed pole spins
every spin has Edwards--Anderson parameter $\E\langle\sigma_x\rangle^2\ge1-O(1/\beta)$, and with
free poles the site overlap $R_{12}=|\V_N|^{-1}\sum_x\sigma_x\tau_x$ ($\V_N$ the vertex set) of
two replicas $\sigma,\tau$ (independent samples from the Gibbs measure with the same couplings)
satisfies $\E\langle R_{12}^2\rangle\ge1-O(1/\beta)$. The blue-cluster statements of
Section~\ref{sec:cmr-intro} also hold. For standard Gaussian couplings, Table~\ref{tab:thresholds}
gives explicit values of $\beta$ above which explicit versions of these lower bounds are positive,
verified in exact rational arithmetic; for $n=4$ this is $\beta\ge520$, that is $T\le1/520$,
compared with $T^*(4)=6/7$ in Theorem~\ref{thm:B}.

Theorem~\ref{thm:B} (computer-assisted) states that for standard Gaussian couplings,
$J_e\sim N(0,1)$, and $n=3,4,5,7,8$ ($n=6$ is covered by Corollary~\ref{cor:Bn} below), at every
$T\le T^*(n)$, where $T^*(n)=4/15,\,6/7,\,6/5,\,7/4,\,2$, the effective coupling grows
geometrically (for each $u$, $\Pp(|\xi_k|\le u)$ decays geometrically in $k$) and the poles are
ordered; and, as $N\to\infty$, the second moment of the site overlap, the expected density of the
blue cluster $C(A)$ of the pole $A$ and the second moment of the blue-cluster imbalance (both
defined in Section~\ref{sec:cmr-intro}) have lower limits bounded below by explicit positive
constants, while no other blue cluster has positive density. Theorem~\ref{thm:B} contains no
fixed-pole statement, and its constants are astronomically small (Table~\ref{tab:points}). A
certificate (a finite set of explicit laws checked by computer, Section~\ref{sec:mechanism}) at
$T^*(n)$ covers every lower temperature, because its only temperature-dependent condition, that
its first law be more peaked than the law of $\beta|J_e|$, becomes easier as $\beta$ grows
(Theorem~\ref{thm:bridge}). The temperatures $T^*(n)$ are roughly $0.89$ to $0.98$ times the
population-dynamics estimates of $T_c(n)$ (a heuristic comparison, not part of the theorem). For
$n=3$ the effective dimension is $\log_26\approx2.585$, and Theorem~\ref{thm:B} shows that, with
standard Gaussian couplings, this lattice has spin-glass order at every $T\le4/15$. Each
certificate also covers every larger number of branches (Corollary~\ref{cor:Bn}): for symmetric
unimodal laws, adding an independent branch makes the effective coupling less peaked, that is,
less concentrated near zero in the peakedness order of Section~\ref{sec:mechanism}, and the map is
monotone in this order. In particular $n=6$ is covered for $T\le6/5$ and every $n\ge9$ for $T\le2$.

For frustrated i.i.d.\ couplings (for instance symmetric laws) with a density, these are to our
knowledge the first proofs of positive-temperature spin-glass order on a Migdal--Kadanoff lattice
with fewer than $10^{18}$ branches; they include $n=4$, the lattice on which the Migdal--Kadanoff
recursion for $\Z^3$ is exact.

\subsection{The mechanism}\label{sec:mechanism}

Both theorems consist of coupling growth followed by one transfer theorem.

\emph{Coupling growth for $n\ge4$} (Section~\ref{sec:partI}). The idea of controlling the
density of the effective coupling at every scale, in order to exclude mass near zero, is due to
Collet and Eckmann \cite[Thm.~4.1(2), Remark~2 and Lemma~A2]{CE1984}, who obtain it from a local
limit theorem in $n$. We replace that step by a one-step contraction valid for small $n$. We
follow a bound $r_k$ on the density of $\xi_k$ restricted to an event $E_k$, determined by the
couplings of one scale-$k$ sub-diamond and of probability at least $1-w_k$. A branch is
\emph{good} if both of its halves lie in their events and their moduli are separated: the larger
is at least $R$ and exceeds the smaller by at least $L$. Given the half with the larger modulus,
the other half is confined to the range below it, of probability $W$, and there has density at
most $r_k/(\Pp(E_k)W)$. At zero temperature the branch value is that smaller half up to sign. Up
to the cut $L$, $W$ is the distribution function of $|\xi_k|$ given $E_k$, evaluated at the larger
of the two moduli; as the larger of two independent uniform variables, $W$ has density $2v$ on $[0,1]$. Averaging over the larger half the uniform law of width $\Pp(E_k)W/r_k$ that matches this
density bound, with the weight $\Pp(E_k)^2$ of both halves lying in their events, therefore gives
at most the triangle $r_k\,p_2(r_kx)$, $p_2(x)=2(1-2|x|)_+$: the factor $1/W$ in the density bound
is cancelled by the probability $W$ of the range. Rogozin's inequality \cite{Rogozin1987}, which bounds the density of a sum
of independent variables by that of a sum of uniform variables with the same density bounds,
applied conditionally on the larger halves, then bounds the density of the sum over $g$ good
branches by $r_k\kappa_g$, $\kappa_g=p_2^{*g}(0)$. Exactly,
\[
\kappa_1=2,\qquad\kappa_2=\tfrac43,\qquad\kappa_3=\tfrac{11}{10},\qquad\kappa_4=\tfrac{302}{315}<1,
\]
and $\kappa_g$ is nonincreasing in $g$. Branches that are not good only shift the sum. A branch
fails to be good with probability at most $\varpi=2w_k+O(r_k)$, and the new event $E_{k+1}$ is
that some branch is good, so $w_{k+1}=\varpi^n$ and $r_{k+1}=\Gamma r_k$ with
$\Gamma=\vartheta\sum_g\binom ng\kappa_g\varpi^{n-g}$, where the factor $\vartheta\ge1$ accounts
for positive temperature and tends to $1$ as $L,R\to\infty$ (Proposition~\ref{prop:A}). For
$n\ge4$, small $\varpi$ and large $L,R$, $\Gamma$ is close to $\kappa_n<1$. This is where
$\beta\ge\beta_0$ enters: once $r_0=\phi/\beta$ is small ($\phi$ bounds the density of $J_e$),
$r_k$ and $w_k$ decay geometrically, so $\sum_k(r_k+w_k)=O(1/\beta)$; since
$\Pp(|\xi_k|<t)\le w_k+2r_kt$, this is the small-ball input of the transfer theorem below, and it
forces $|\xi_k|\to\infty$ in probability. At zero temperature the bound $r_k\kappa_g$ is attained
and $\kappa_3=11/10>1$, which is why $n=3$ is out of reach of one-step bounds on the supremum of
the density.

\emph{Coupling growth near $T_c$} (Section~\ref{sec:partII}). For symmetric unimodal inputs the
exact one-step map preserves symmetric unimodality and is monotone in the peakedness order, in
which $X$ is more peaked than $Y$ if $\Pp(|X|\le t)\ge\Pp(|Y|\le t)$ for all $t\ge0$
(Lemma~\ref{lem:U} and Corollary~\ref{cor:monotone}). A finite chain of explicit staircase laws
(an atom at $0$ plus a nonincreasing piecewise-constant density on a uniform grid, so that the
distribution function is concave and piecewise linear), each more
peaked than the true law at its level, is checked by exact and interval arithmetic until the
law enters a scale-covariant family $\{L_0\lambda^iX^{(\alpha_i)}\}_{i\ge0}$; here $X^{(\alpha)}$ is
$0$ with probability $\alpha$ and has a fixed shape law $X_c$ otherwise, $L_0$ is a scale (unrelated to
the cut $L$ above), and
$\alpha_i\to0$. The invariance of the family under the map, with growth factor $\lambda>1$, is
itself a finite check: it reduces to a finite set of inequalities between explicit
staircase laws, checked once, and Theorem~\ref{thm:family} propagates it to every level.
This gives $|\xi_k|\to\infty$ geometrically.

\emph{Transfer to the bulk} (Section~\ref{sec:transfer}). A spin $x$ of level $j$, the midpoint
of a branch of a scale-$j$ sub-diamond, sits in a chain of nested sub-diamonds $D_N\supset\dots\supset D_j$. Given the pole spins of $D_k$,
the midpoint $y_k$ of the branch leading to $x$ feels the field
$\xi^{(1)}\sigma_a+\xi^{(2)}\sigma_b$ from the two halves of that branch; this is the exact
single-site recursion of \cite{MCC1990,NCNCA1997}. Conditional on the pole spins in both replicas, when the endpoint overlaps
have a common value $q$, the midpoint overlap differs from $q$ with probability
at most $\frac12\sech^2\Delta_k$,
$\Delta_k=||\xi^{(1)}|-|\xi^{(2)}||$.
A union bound over the ancestor chain gives order at $x$
(Lemma~\ref{lem:S}). The architecture---an ancestor
chain, a failure event governed by the gap between the halves, small-ball input from a density
bound, a sum over levels---is that of Collet and Eckmann \cite[\S VII, Lemma~7.8]{CE1984};
we use it in a two-replica, union-bound form that needs no symmetry of the coupling law, applies
to free and fixed poles alike, and takes as input any small-ball profile, that is, bounds on
$\Pp(|\xi_k|<t)$ and on $\sup_y\Pp(||\xi_k|-y|<t)$ for $k\ge0$, $t>0$
(Theorem~\ref{thm:transfer}). The order statements are proved without the cluster representation
of Section~\ref{sec:cmr-intro}.

\subsection{Blue clusters and imbalance}\label{sec:cmr-intro}

In the Chayes--Machta--Redner (CMR) representation \cite{CMR1998} of two replicas, an edge
$e=xy$ that is satisfied ($J_e\sigma_x\sigma_y>0$) in both replicas is blue with probability
$1-e^{-4\beta|J_e|}$, and the overlap $q_x=\sigma_x\tau_x$ is constant on blue clusters (compare
the signed-connectivity identity of Machta, Newman and Stein \cite[eq.~(16)]{MNS2008}); blue bonds
enter only the cluster statements. Machta, Newman and Stein \cite{MNS2008} proposed that
spin-glass order appears as two percolating blue clusters of unequal densities, and proved for
the Sherrington--Kirkpatrick model that the two largest blue clusters have densities
$(1\pm|R_{12}|)/2$, $R_{12}$ the overlap of the two replicas, up to errors that vanish in
probability \cite[Thm.~1]{MNS2007}. One additional ingredient is a two-terminal crossing lemma
(Lemma~\ref{lem:X}): if neither of two vertices of a graph has its spin fixed, the probability
that they are not blue-connected is at most $1-\tanh|\xi|$, where $\xi$ is the effective coupling
between them. With it, the same coupling growth gives the following on $\DL_N$. Under
Theorem~\ref{thm:A}(iv), with probability bounded below by a positive constant uniformly in $N$,
the blue cluster $C(A)$ of the pole $A$ contains both poles and has density bounded below by a positive constant, and $C(A)$ contains both poles with probability tending to one, since $\Pp(A\not\leftrightarrow B)\le P_N\le w_N+2\ln2\,r_N\to0$ (Theorem~\ref{thm:transfer}, item~4, and Lemma~\ref{lem:ratesH}); under Theorem~\ref{thm:B}(i),(iii), $C(A)$ contains both poles with
probability tending to one and its expected density has a positive lower limit as $N\to\infty$.
In both cases no other blue cluster has positive density, in probability and in $L^1$, and the
imbalance $\E[((|C_+|-|C_-|)/|\V_N|)^2]$, where $C_\pm$ are the largest blue clusters with
overlap $\pm1$, is bounded away from zero (uniformly in $N$ under Theorem~\ref{thm:A}, as a lower
limit under Theorem~\ref{thm:B}). On the hierarchical lattice this imbalance is degenerate: one
giant pole cluster plus sub-extensive dust. A blue cluster that avoids the poles and has more
than $(2n)^{-m}|\V_N|$ vertices must contain one of the $v_m$ vertices of level above $N-m$, a
number independent of $N$ (Lemma~\ref{lem:struct}(d)), and these are blue-connected to the poles
with high probability (Theorem~\ref{thm:transfer}, item~5). So only one of the two percolating
clusters of the Machta--Newman--Stein picture appears. To our knowledge this is the first proof of
blue-cluster imbalance, in this sense, for frustrated couplings (for instance symmetric laws) on a
graph other than the complete graph. It is a statement about hierarchical lattices. On $\Z^d$, in
local limits of periodic-torus measures, the blue subgraph has at most two infinite clusters
\cite{Pei2026b}, and blue percolation has been proved in high dimensions \cite{Pei2026}, but
imbalance remains open.\footnote{An earlier claim, by the author and Di Ventra,
of a finite-temperature spin-glass transition in $d\ge2$ via blue clusters
(\href{https://arxiv.org/abs/2105.01188}{arXiv:2105.01188})
was withdrawn by its authors (v4, 19 September 2022) after an error was found in its contour
estimate.}

\subsection{Organization}\label{sec:organization}

Section~\ref{sec:model} defines the model and states Theorems~\ref{thm:A} and~\ref{thm:B}.
Section~\ref{sec:recursion} collects the structure of the lattice, the exact recursion, the
Markov structure and the crossing lemma. Section~\ref{sec:transfer} proves the transfer theorem,
which turns small-ball bounds on the effective couplings into the bulk statements of both
theorems; Sections~\ref{sec:partI} and~\ref{sec:partII} supply these bounds.
Section~\ref{sec:partI} proves coupling growth for $n\ge4$ and Theorem~\ref{thm:A};
Section~\ref{sec:partII} proves coupling growth from certificates and Theorem~\ref{thm:B};
Section~\ref{sec:computation} states what the computer verifies. Section~\ref{sec:discussion}
discusses open problems. Appendix~\ref{app:proofs} contains the proofs of routine lemmas.
Appendix~\ref{app:computation} describes the certificates for Theorem~\ref{thm:A}, proves the
reduction of the bridge condition to pairwise sums (Lemma~\ref{lem:decomp}), and gives the
arithmetic, the checks, the validation and the reproduction of the computation.

\section{Model and results}\label{sec:model}

\subsection{The diamond lattice}\label{sec:lattice}

Fix $n\ge1$. Let $\DL_0$ be the graph with two vertices $A,B$ and one edge. For $N\ge1$, $\DL_N$
consists of $n$ \emph{branches} between $A$ and $B$; each branch is two copies of $\DL_{N-1}$,
its \emph{halves}, glued in series at a new vertex, the \emph{midpoint} of the branch. We call
$N$ the \emph{level} of $\DL_N$. Every copy
of $\DL_k$ that arises in this construction is a \emph{scale-$k$ sub-diamond} $D$, with poles
$a_D,b_D$ and interior $\operatorname{int}D:=V(D)\setminus\{a_D,b_D\}$. A scale-$k$ sub-diamond
with $k\ge1$ has $n$ branches, and the two halves of a branch are edge-disjoint scale-$(k-1)$
sub-diamonds. A vertex $x\notin\{A,B\}$ has \emph{level} $k_0(x)=j$ if it is the midpoint of a
branch of a scale-$j$ sub-diamond, denoted $D(x)$.

Write $\V_N=V(\DL_N)$ and $v_k:=|\operatorname{int}\DL_k|$. Then $|E(\DL_N)|=(2n)^N$,
$v_0=0$, $v_k=n(1+2v_{k-1})$, so $v_N=n((2n)^N-1)/(2n-1)$ and $|\V_N|=v_N+2$. The number of
level-$j$ vertices is $N_j=n(2n)^{N-j}$ ($1\le j\le N$), and exactly
\begin{equation}\label{eq:counts}
\frac{N_j}{|\V_N|}=(2n-1)(2n)^{-j}(1-\varrho_N),\qquad
\varrho_N:=\frac{3n-2}{n(2n)^N+3n-2}.
\end{equation}
The graph distance between the poles is $2^N$ and the number of edges is $2^{Nd_\eff}$ with
$d_\eff=1+\log_2n$.

\subsection{Gibbs measures, replicas and blue bonds}\label{sec:measures}

The couplings $J=(J_e)_{e\in E(\DL_N)}$ are independent and identically distributed real
random variables. At inverse temperature $\beta\ge0$ the \emph{free-boundary} (or
\emph{free-pole}) Gibbs measure on
$\{-1,1\}^{\V_N}$ is $\mu^{\beta,J}_N(\sigma)\propto\exp(\beta\sum_{e=xy}J_e\sigma_x\sigma_y)$,
with all spins, including $\sigma_A$ and $\sigma_B$, summed. The \emph{fixed-pole} measure with
pole spins $(s_A,s_B)\in\{-1,1\}^2$ is $\mu^{\beta,J}_N$ conditioned on $\sigma_A=s_A$,
$\sigma_B=s_B$. We write $\langle\cdot\rangle$ for the Gibbs expectation, and
$\langle\cdot\rangle_{s_A,s_B}$ when the pole spins are fixed. Two \emph{replicas}
$\sigma,\tau$ are independent given $J$ and have the same Gibbs measure. Put
$q_x:=\sigma_x\tau_x$ and $R_{12}:=|\V_N|^{-1}\sum_{x\in\V_N}q_x$. We write $\Pp,\E$ for the
joint law of couplings, replicas and bonds. For all vertices $x,y$ and either boundary condition,
\begin{equation}\label{eq:twopoint}
\E\langle\sigma_x\sigma_y\rangle^2=\E[q_xq_y],\qquad \E\langle\sigma_x\rangle^2=\E[q_x],
\end{equation}
because $\langle\sigma_x\sigma_y\rangle^2=\langle\sigma_x\sigma_y\rangle\langle\tau_x\tau_y\rangle$.

\emph{Blue bonds} \cite{CMR1998,MNS2008}. Given $(J,\sigma,\tau)$, each edge $e=xy$ is
independently \emph{blue} with probability $1-e^{-4\beta|J_e|}$ if $J_e\sigma_x\sigma_y>0$ and
$J_e\tau_x\tau_y>0$, and is not blue otherwise. Summing over bonds returns the two-replica
measure. A blue edge has $\sigma_x\sigma_y=\tau_x\tau_y$, so $q$ is constant on blue clusters.
Write $x\leftrightarrow y$ if $x$ and $y$ are joined by blue edges, and $C(x)$ for the blue
cluster of $x$. Let $C_\pm$ be a largest blue cluster with $q\equiv\pm1$ (empty if there is
none), let
\begin{equation}\label{eq:imb}
\Imb_N:=\frac{|C_+|-|C_-|}{|\V_N|}
\end{equation}
be the \emph{imbalance}, and let $M_N:=\max\{|C|:C\text{ a blue cluster},\ A\notin C\}$ ($0$ if
there is none).

\subsection{Effective couplings}\label{sec:effective}

For a sub-diamond $D$ and $s,s'\in\{-1,1\}$ let $Z_D(s,s')$ be the Boltzmann sum over the
interior spins of $D$ with pole spins $(s,s')$ and couplings of $D$ only, and
$\xi_D:=\frac12\log(Z_D(+,+)/Z_D(+,-))$. We write $\xi_k$ for the effective coupling of a
scale-$k$ sub-diamond; its law does not depend on $N$ and satisfies \eqref{eq:intro-rec}
(Lemma~\ref{lem:rec}).

\subsection{Theorem A: analytic order for \texorpdfstring{$n\ge4$}{n ≥ 4}}\label{sec:thmA}

We assume
\begin{itemize}
\item[\textup{(D)}] the $J_e$ are independent and identically distributed, and their law has a
density bounded by $\phi<\infty$.
\end{itemize}
No symmetry and no moment condition is assumed. The proof follows two numbers per level: $r_k$
bounds the density of $\xi_k$ on an event $E_k$ determined by the couplings of one scale-$k$
sub-diamond, and $w_k$ bounds $\Pp(E_k^c)$ (Definition~\ref{def:H}). A \emph{certificate}
(Definition~\ref{def:certA}) is a sequence $(r_k,w_k)_{k\ge0}$ of such bounds, propagated level by
level by Proposition~\ref{prop:A} and, from some level on, continued geometrically by
Lemma~\ref{lem:T} (the certificate \emph{closes} there). Despite the name it is an analytic
object: Lemma~\ref{lem:E} constructs one for every large $\beta$, and a computer is used only to
find and verify explicit certificates for standard Gaussian couplings (Tables~\ref{tab:thresholds}
and~\ref{tab:oneshot}). The constants $\ell_{\rm s},\ell_{\rm b},S,P$ of a certificate, defined
in \eqref{eq:consts}, sum or maximize per-level bounds over the levels: $\ell_{\rm s}$ bounds the
probability that the overlap at a non-pole vertex differs from a common value of the overlap at
the poles, $\ell_{\rm b}$ the probability that a vertex lies in neither blue pole cluster, and
$S\ge\E\sech^2\xi_N$ and $P\ge\Pp(A\not\leftrightarrow B)$ are the corresponding pole terms
(Section~\ref{sec:transfer}). The per-level rate of the blue chain is bounded
(Lemma~\ref{lem:ratesH}) by $\bar a^{\rm b}(r_{k-1},w_{k-1})$, where
\begin{equation}\label{eq:abarb}
\bar a^{\rm b}(r,w):=\tfrac34w+(1+2\ln2)\,r+w^2+wr+\tfrac12r^2 .
\end{equation}

\begin{mainthm}[analytic]\label{thm:A}
Let $n\ge4$ and assume \textup{(D)}. There is $\beta_0=\beta_0(n,\phi)<\infty$ such that for every
$\beta\ge\beta_0$ there is a certificate $(r_k,w_k)_{k\ge0}$ whose constants
$\ell_{\rm s},\ell_{\rm b},S,P\in[0,1)$ of \eqref{eq:consts} satisfy $2\ell_{\rm s}+S<1$ and
$2\ell_{\rm b}+\sqrt P<1$ (equivalently $(1-2\ell_{\rm b})_+^2>P$), and the following hold, the
inequalities for every $N\ge1$.
\begin{enumerate}
\item[\textup{(i)}] \emph{Pole order.} With free poles, $\E\langle\sigma_A\sigma_B\rangle^2\ge1-S$,
and $\E\langle\sigma_A\sigma_B\rangle^2\to1$ as $N\to\infty$.
\item[\textup{(ii)}] \emph{Edwards--Anderson order with fixed poles.} For all
$s_A,s_B\in\{-1,1\}$ and every $x\in\V_N\setminus\{A,B\}$,
$\E\langle\sigma_x\rangle_{s_A,s_B}^2\ge1-2\ell_{\rm s}$. If the law of $J_e$ is symmetric,
then also $\E\langle\sigma_x\rangle_{s_A,s_B}=0$.
\item[\textup{(iii)}] \emph{Site-overlap order.} With free poles,
$\E\langle R_{12}^2\rangle\ge(1-2\ell_{\rm s}-S)_+^2>0$.
\item[\textup{(iv)}] \emph{A giant blue cluster containing the poles.} With free poles, for every
$\theta\in(0,1)$,
\[
\Pp\big(B\in C(A),\ |C(A)|>(1-\theta)|\V_N|\big)\ge1-\ell_{\rm b}/\theta-P,
\]
which is positive for $\theta\in(\ell_{\rm b}/(1-P),1)$.
\item[\textup{(v)}] \emph{At most one giant blue cluster.} With free poles, for $1\le m\le N$,
\[
\Pp\big(M_N>(2n)^{-m}|\V_N|\big)\le\epsilon_{N,m},\qquad
\E M_N/|\V_N|\le(2n)^{-m}+\epsilon_{N,m},
\]
where $\epsilon_{N,m}:=w_N+2\ln2\,r_N+v_m\sum_{k=N-m+1}^{N}\bar a^{\rm b}(r_{k-1},w_{k-1})$ tends to
$0$ as $N\to\infty$ for each fixed $m$. Hence $M_N/|\V_N|\to0$ in probability and in $L^1$.
\item[\textup{(vi)}] \emph{Blue imbalance.} With free poles,
$\E\,\Imb_N^2\ge(1-2\ell_{\rm b})_+^2-P>0$.
\end{enumerate}
The constants $\ell_{\rm s},\ell_{\rm b},S,P$ are $O(1/\beta)$ as $\beta\to\infty$ for fixed
$(n,\phi)$, so the lower bounds in \textup{(i)--(iii)} and \textup{(vi)} tend to $1$.
Items \textup{(i)--(vi)} hold, with the constants of the certificate, at every $\beta$ at which a
certificate exists, except that the positivity assertions in \textup{(iii)} require
$2\ell_{\rm s}+S<1$ and those in \textup{(iv)} and \textup{(vi)} require $2\ell_{\rm b}+\sqrt P<1$;
and a certificate at $\beta_1$ is a certificate at every $\beta\ge\beta_1$.
\end{mainthm}

Item (ii) is the analogue for every $n\ge4$ of \cite[Thm.~7.1]{CE1984}; the hypotheses of that
result are recalled in Section~\ref{sec:problem}.

For standard Gaussian couplings ($\phi=1/\sqrt{2\pi}$), Table~\ref{tab:thresholds} lists certified
thresholds, and Table~\ref{tab:oneshot} lists certificates for pole order in which
Lemma~\ref{lem:T} applies already at level $0$; they can be checked by hand
(Appendix~\ref{sec:gauss-certs}).

\begin{table}[t]
\centering
\begin{tabular}{rrrr}
\toprule
$n$ & $\beta_{\rm pole}$ & $\beta_{\rm order}$ & $\beta_{\rm imb}$\\
\midrule
4 & 520.0 & 520.0 & 520.0 \\
5 & 132.2 & 132.2 & 132.3 \\
6 & 84.7 & 84.8 & 84.8 \\
7 & 67.0 & 67.0 & 67.0 \\
8 & 57.8 & 57.8 & 57.8 \\
12 & 43.0 & 43.1 & 43.1 \\
16 & 35.7 & 35.8 & 36.0 \\
32 & 12.5 & 13.5 & 14.8 \\
64 & 9.3 & 9.9 & 10.8 \\
\bottomrule
\end{tabular}
\caption{Theorem~\ref{thm:A} for standard Gaussian couplings. Each entry is the $\beta$ on the
grid $0.1\,\Z$ at which a certificate was found and verified in exact rational arithmetic with
rigorous enclosures (Appendix~\ref{sec:gauss-certs}); the search failed at the next grid point
below, which does not prove that no certificate exists there. Columns: $\beta_{\rm pole}$, a
certificate exists, so (i) and the upper bound (v) hold; $\beta_{\rm order}$, in addition
$2\ell_{\rm s}+S<1$, so (ii) and (iii) hold with positive bounds; $\beta_{\rm imb}$, in addition
$2\ell_{\rm b}+\sqrt P<1$, so (iv) and (vi) hold with positive bounds. Each entry certifies every
larger $\beta$.}
\label{tab:thresholds}
\end{table}

\subsection{Theorem B: computer-assisted order near \texorpdfstring{$T_c$}{T\_c}}\label{sec:thmB}

Put $c_0:=\frac12\ln2$, the constant in the bound $a\star b\ge(\min(a,b)-c_0)_+$ of
Lemma~\ref{lem:series}(E1).

\begin{mainthm}[computer-assisted]\label{thm:B}
Let the $J_e$ be independent standard Gaussian variables and let $(n,T^*)$ be one of
\[
(3,\tfrac4{15}),\quad(4,\tfrac67),\quad(5,\tfrac65),\quad(7,\tfrac74),\quad(8,2).
\]
Let $0<T\le T^*$ and $\beta=1/T$. With free poles the following hold.
\begin{enumerate}
\item[\textup{(i)}] \emph{Coupling growth.} There are an integer $J$, the number of bridge steps
(Definition~\ref{def:bridge}; Table~\ref{tab:points} lists it), a growth factor $\lambda>1$ and
constants $L_0,A,f_c>0$ of the certificate (Definition~\ref{def:certB}; this constant $A$ is not
the pole $A$), such that
\[
\Pp(|\xi_{J+i}|\le u)\le\frac{Ac_0+f_cu}{L_0\lambda^i}\qquad(i\ge0,\ u\ge0).
\]
In particular $|\xi_N|\to\infty$ in probability, $\E\langle\sigma_A\sigma_B\rangle^2\to1$ and
$\Pp(A\leftrightarrow B)\to1$.
\item[\textup{(ii)}] \emph{Site-overlap order.} $\liminf_{N\to\infty}\E\langle R_{12}^2\rangle\ge\kappa_Q^2$.
\item[\textup{(iii)}] \emph{A unique giant pole cluster.}
$\liminf_{N\to\infty}\E|C(A)|/|\V_N|\ge\kappa_C$ and $\E M_N/|\V_N|\to0$.
\item[\textup{(iv)}] \emph{Blue imbalance.} $\liminf_{N\to\infty}\E\,\Imb_N^2\ge\kappa_C^2$.
\end{enumerate}
Here $\kappa_Q,\kappa_C>0$ are defined in \eqref{eq:kappas} and satisfy the lower bounds of
Table~\ref{tab:points}. Bounds valid for each finite $N$ follow from
Theorem~\ref{thm:transfer} with the rates of Section~\ref{sec:proofB}.
\end{mainthm}

Theorem~\ref{thm:B} rests on a rigorous computer verification, in exact rational and interval
arithmetic, of finitely many inequalities between explicit laws (Section~\ref{sec:computation} and
Appendix~\ref{app:computation}); every other step is proved in
Sections~\ref{sec:recursion}--\ref{sec:partII} and Appendices~\ref{app:proofs}
and~\ref{app:decomp}.

\begin{table}[t]
\centering
\begin{tabular}{rccrrr}
\toprule
$n$ & $T^*$ & $\beta^*$ & $J$ & $\log_{10}\kappa_Q\ge$ & $\log_{10}\kappa_C\ge$\\
\midrule
3 & $4/15$ & $15/4$ & 251 & $-199.009$ & $-155.225$ \\
4 & $6/7$ & $7/6$ & 43 & $-33.094$ & $-28.135$ \\
5 & $6/5$ & $5/6$ & 26 & $-19.740$ & $-16.255$ \\
7 & $7/4$ & $4/7$ & 19 & $-15.641$ & $-13.427$ \\
8 & $2$ & $1/2$ & 19 & $-16.435$ & $-14.751$ \\
\bottomrule
\end{tabular}
\caption{The computer-assisted certificates of Theorem~\ref{thm:B} (standard Gaussian couplings).
$T^*$ and $\beta^*=1/T^*$ are exact; $J$ is the index of the last bridge law, so the bridge has
$J$ steps (Definition~\ref{def:bridge}); the last two columns are lower bounds for $\log_{10}\kappa_Q$ and
$\log_{10}\kappa_C$, rounded down to three decimals. Each certificate covers every $T\le T^*$.}
\label{tab:points}
\end{table}

The certificates also cover larger numbers of branches (proof in Section~\ref{sec:proofB}).

\begin{corollary}[More branches]\label{cor:Bn}
Let $n'\ge3$ and let $(n,T^*)$ be one of the five pairs of Theorem~\ref{thm:B} with $n\le n'$. For
standard Gaussian couplings on the lattice with $n'$ branches and every $0<T\le T^*$, items
\textup{(i)--(iv)} of Theorem~\ref{thm:B} hold with the same constants $J,\lambda,L_0,A,f_c$ in
\textup{(i)}, and with $\kappa_Q,\kappa_C$ replaced by the positive constants \eqref{eq:kappas}
formed with $n'$ in place of $n$ and the same certified rate sums $\bar T_j$ of
Section~\ref{sec:proofB}. In particular the conclusions hold for
$n'=6$ at every $T\le6/5$ and for every $n'\ge9$ at every $T\le2$.
\end{corollary}

\subsection{Remarks on the statements}\label{sec:remarks}

\begin{remark}[Boundary conditions]\label{rem:bc}
The spin statements rest on Lemma~\ref{lem:S}, which holds for free poles, fixed poles, one fixed
pole, and mixtures (Section~\ref{sec:freefixed}). The blue statements use
$\Pp(A\not\leftrightarrow B)\le\E(1-\tanh|\xi_N|)$, which holds for free poles and for fixed poles
in the preferred configuration (pole spins whose product has the sign of $\xi_N$,
Section~\ref{sec:crossing}); for fixed poles of the other sign this bound is not available
(Section~\ref{sec:freefixed}).
\end{remark}

\begin{remark}[Scope]\label{rem:scope}
\begin{enumerate}
\item[(a)] Theorem~\ref{thm:A} is of interest for frustrated laws, in particular laws with mean
zero; for strongly biased laws the order is ferromagnetic.
\item[(b)] Both theorems assume a density; whether $\pm J$ couplings order on these lattices is
open (Sections~\ref{sec:problem} and~\ref{sec:discussion}).
\item[(c)] Only the pole cluster $C(A)$ is giant. That $C(A)$ has density at least $1-\theta$ with
probability close to one is proved only for large $\beta$: in Theorem~\ref{thm:A}(iv) the bound
$1-\ell_{\rm b}/\theta-P$ tends to $1$ as $\beta\to\infty$. In Theorem~\ref{thm:B} only the lower
limit of its expected density as $N\to\infty$ is bounded below. In the Sherrington--Kirkpatrick
model there are two giant clusters \cite[Thm.~1]{MNS2007}.
\item[(d)] The temperatures $T^*(n)$ are limits of the present implementation (grids, schedules,
parameters); we do not know how close to $T_c(n)$ the method can reach. The constants
$\kappa_Q,\kappa_C$ are astronomically small because the transfer uses a union bound over all
levels of the ancestor chain (Lemmas~\ref{lem:S} and~\ref{lem:Bprime}), and this bound is
dominated by the levels near the critical scale, before the coupling has grown: in
\eqref{eq:kappas} a level $j$ contributes only once $2\bar T_j<1$, respectively $\bar T_j<1$, and
the vertices of level at least $j$ form a fraction of order $(2n)^{-j}$ of all vertices by
\eqref{eq:counts}.
\item[(e)] Theorem~\ref{thm:A} is restricted to $n\ge4$ because the one-step constant
$\kappa_3=11/10$ exceeds $1$ (Remark~\ref{rem:sharp}).
\end{enumerate}
\end{remark}

\section{The exact recursion and two-replica facts}\label{sec:recursion}

This section collects the facts used by the transfer theorem: the ancestor chain of a vertex
(Section~\ref{sec:structure}), the exact recursion and the series map (Section~\ref{sec:rec}), the
Markov structure of two replicas given pole spins (Section~\ref{sec:markov}), and the crossing
lemma, which turns a large effective coupling between two vertices into a blue connection between
them and is proved in Section~\ref{sec:crossing}. The proofs of Lemmas~\ref{lem:struct},
\ref{lem:rec}, \ref{lem:series} and~\ref{lem:markov} are routine and are given in
Appendix~\ref{app:proofs}.

\subsection{Structure of the lattice}\label{sec:structure}

\begin{lemma}[Structure]\label{lem:struct}
\begin{enumerate}
\item[\textup{(a)}] Every edge incident to a vertex of $\operatorname{int}D$ is an edge of $D$.
\item[\textup{(b)}] Two sub-diamonds are nested or have disjoint interiors.
\item[\textup{(c)}] The effective couplings of edge-disjoint sub-diamonds, for example the two
halves of a branch, are independent.
\item[\textup{(d)}] For $0\le m\le N$ let $W_m:=\{x\in\V_N\setminus\{A,B\}:k_0(x)\ge N-m+1\}$.
Then $|W_m|=v_m$, every connected set of vertices that avoids $W_m\cup\{A,B\}$ lies in the
interior of one scale-$(N-m)$ sub-diamond, and $v_{N-m}\le(2n)^{-m}|\V_N|$.
\end{enumerate}
\end{lemma}

\emph{Ancestors.} Let $x$ have level $j$. For $j\le k\le N$ let $D_k=D_k(x)$ be the unique
scale-$k$ sub-diamond with $x\in\operatorname{int}D_k$ (unique by Lemma~\ref{lem:struct}(b));
$D_N=\DL_N$ and $D_j=D(x)$. For $k>j$, $x$ is not a midpoint of a branch of $D_k$, so it lies in
the interior of exactly one half of one branch of $D_k$, and that half is $D_{k-1}$. Let $y_k$ be
the midpoint of this branch ($k>j$), and $y_j:=x$, the midpoint of a branch of $D_j$; let
$H^{(1)}_k$ (poles $a_{D_k},y_k$) and $H^{(2)}_k$ (poles $y_k,b_{D_k}$) be the halves of the
branch with midpoint $y_k$, and let $\xi^{(1)}_k,\xi^{(2)}_k$ be their effective couplings, which
are independent copies of $\xi_{k-1}$ (Lemma~\ref{lem:struct}(c)). Put
\begin{equation}\label{eq:gap}
\Delta_k:=\big||\xi^{(1)}_k|-|\xi^{(2)}_k|\big|,\qquad \Xi_k:=\max(|\xi^{(1)}_k|,|\xi^{(2)}_k|),
\end{equation}
let $H^*_k$ be the stronger half (larger $|\xi|$; the first one in case of a tie), and let
$\pi_k$ be the pole of $D_k$ that $H^*_k$ contains. Since $D_{k-1}$ is a half of the branch with
midpoint $y_k$,
\begin{equation}\label{eq:nesting}
\{a_{D_{k-1}},b_{D_{k-1}}\}\subset\{a_{D_k},b_{D_k},y_k\}\qquad(j<k\le N).
\end{equation}

\subsection{The exact recursion and the series map}\label{sec:rec}

\begin{lemma}[Exact recursion]\label{lem:rec}
\begin{enumerate}
\item[\textup{(a)}] $Z_D(s,s')=Z_D(-s,-s')$, so that $Z_D(s,s')=C_De^{\xi_Dss'}$ for a constant $C_D>0$.
\item[\textup{(b)}] If $D$ consists of $D_1,D_2$ in series, $\tanh\xi_D=\tanh\xi_{D_1}\tanh\xi_{D_2}$,
i.e.\ $\xi_D=\xi_{D_1}\star\xi_{D_2}$; if $D$ consists of $D_1,\dots,D_n$ in parallel,
$\xi_D=\sum_i\xi_{D_i}$. Hence \eqref{eq:intro-rec} holds with $\xi_0=\beta J_e$.
\item[\textup{(c)}] With free poles, $\langle\sigma_A\sigma_B\rangle=\tanh\xi_N$ and
$\Pp(q_A\ne q_B\mid J)=\frac12\sech^2\xi_N$.
\item[\textup{(d)}] If $y$ is the midpoint of a branch of $D$ with halves $H^{(1)}$ (poles $a,y$)
and $H^{(2)}$ (poles $y,b$), then under the fixed-pole Gibbs measure of $D$ with pole spins
$\sigma_a,\sigma_b$ one has
$\Pp(\sigma_y=s)\propto\exp\big(s(\xi_{H^{(1)}}\sigma_a+\xi_{H^{(2)}}\sigma_b)\big)$.
\end{enumerate}
\end{lemma}

Lemma~\ref{lem:rec}(d) is the exact single-site recursion
$\langle\sigma_y\rangle=\tanh(\xi^{(1)}\sigma_a+\xi^{(2)}\sigma_b)$, conditionally on the spins at
$a,b$; it underlies the recursion for local magnetizations of \cite{MCC1990} and
\cite[eq.~(11)]{NCNCA1997}, and the chains of \cite[\S VII]{CE1984}.

\begin{lemma}[The series map]\label{lem:series}
For $a,b\ge0$ write $a\star b=\artanh(\tanh a\tanh b)$, so that $e^{2(a\star b)}=\cosh(a+b)/\cosh(a-b)$.
\begin{enumerate}
\item[\textup{(E0)}] $\star$ is symmetric, nondecreasing in each argument, and $a\star b\le\min(a,b)$.
\item[\textup{(E1)}] For $0\le a\le b$,
$a\star b=a+\frac12\log(1+e^{-2(a+b)})-\frac12\log(1+e^{-2(b-a)})$; hence
$a\star b\ge(\min(a,b)-c_0)_+$ with $c_0=\frac12\ln2$.
\item[\textup{(E2)}] For $u\ge0$, $\{a\star b\le u\}=\{b\le u\}\cup\{b>u,\ a\le\iota(u,b)\}$, where
for $b>u$
\[
\iota(u,b):=\artanh\frac{\tanh u}{\tanh b}=u+\delta(b-u)-\delta(b+u),\qquad
\delta(d):=-\tfrac12\log(1-e^{-2d}),
\]
and $\iota$ is decreasing in $b$. Hence, for independent $a,b\ge0$ with $a$ having CDF $F_a$,
$\Pp(a\star b\le u)=\Pp(b\le u)+\E[\1\{b>u\}F_a(\iota(u,b))]$.
\end{enumerate}
\end{lemma}

Collet and Eckmann use the weaker bound $a\star b\ge\min(a,b)-\ln2$ \cite[Lemma~A1.2]{CE1984}.

\subsection{Markov structure}\label{sec:markov}

\begin{definition}\label{def:admissible}
An \emph{admissible boundary law} is a law of the four pole spins $(\sigma_A,\sigma_B,\tau_A,\tau_B)$,
possibly depending on $J$, such that given these spins (and $J$) the two replicas are independent
and each has the fixed-pole Gibbs law; blue bonds are then added given $(J,\sigma,\tau)$ as in
Section~\ref{sec:measures}. Free poles (two independent free replicas), fixed poles
($\sigma_A=\tau_A=s_A$, $\sigma_B=\tau_B=s_B$), one fixed pole, and mixtures are admissible.
\end{definition}

\begin{lemma}[Markov structure]\label{lem:markov}
Under an admissible boundary law, let $D$ be a sub-diamond. Conditionally on $J$ and on the spins
at $a_D,b_D$ in both replicas:
\begin{enumerate}
\item[\textup{(a)}] the restrictions of $\sigma$ and $\tau$ to $\operatorname{int}D$ are independent,
each with the fixed-pole Gibbs law of $D$, and the blue bonds of the edges of $D$ are
conditionally independent given the spins, with the probabilities of
Section~\ref{sec:measures};
\item[\textup{(b)}] if $y$ is the midpoint of a branch of $D$ with halves $H^{(1)},H^{(2)}$, then
\[
\Pp(\sigma_y=s\mid\cdot)\propto\exp\big(s(\xi_{H^{(1)}}\sigma_{a_D}+\xi_{H^{(2)}}\sigma_{b_D})\big),
\]
the same holds for $\tau$ independently, and conditionally also on $\sigma_y,\tau_y$ the spins and
blue bonds inside $H^{(1)}$ and $H^{(2)}$ are independent with the fixed-pole two-replica laws of
these graphs.
\end{enumerate}
\end{lemma}

\subsection{The crossing lemma}\label{sec:crossing}

Let $G$ be a finite graph with two vertices $u\ne v$, real couplings and $\beta\ge0$. Let $Z(s)$ be
the one-replica Boltzmann weight of $\{\sigma_u=+1,\sigma_v=s\}$ and
$\xi_G:=\frac12\log(Z(+)/Z(-))$ (for a sub-diamond, $\xi_G=\xi_D$). A pole configuration is
\emph{preferred} if $\sigma_u\sigma_v=\tau_u\tau_v=\sgn\xi_G$ (with $\sgn0:=+1$).

\begin{lemma}[Crossing]\label{lem:X}
Let $\mathcal N_s$ be the two-replica CMR weight of
$\{\sigma_u=\tau_u=+1,\ \sigma_v=\tau_v=s,\ u\not\leftrightarrow v\}$.
\begin{enumerate}
\item[\textup{(a)}] $\mathcal N_+=\mathcal N_-\le\min(Z(+),Z(-))^2$.
\item[\textup{(b)}] Under the fixed-pole two-replica law with a preferred pole configuration,
$\Pp(u\not\leftrightarrow v)\le e^{-4|\xi_G|}$.
\item[\textup{(c)}] With free poles, $\Pp(u\not\leftrightarrow v)\le2/(1+e^{2|\xi_G|})=1-\tanh|\xi_G|$.
\end{enumerate}
\end{lemma}

\begin{proof}
The global flips $\sigma\mapsto-\sigma$ and $\tau\mapsto-\tau$ preserve the two-replica measure,
the law of the blue bonds (which depends on the products $\sigma_x\sigma_y$, $\tau_x\tau_y$),
preferredness and blue connectivity; so it suffices to treat $\sigma_u=\tau_u=+1$.
Relative to $e^{2\beta|J_e|}$, an edge has weight $1-r_e^2$ if it is satisfied in both replicas
and blue, $r_e^2$ if satisfied in both and not blue, $r_e$ if satisfied in exactly one, and $r_e^2$
if satisfied in neither, where $r_e=e^{-2\beta|J_e|}$. On $\{u\not\leftrightarrow v\}$ flip both
replicas on the blue cluster of $v$, keeping the blue edge set. Edges inside the cluster keep their
state; boundary edges are not blue, and the flip exchanges ``satisfied in both'' with ``satisfied in
neither'' (both $r_e^2$) and the two ``satisfied in one'' states (both $r_e$). The map is a
weight-preserving involution between the configurations counted by $\mathcal N_+$ and by
$\mathcal N_-$; this is the blue-cluster flip (``blue moves'') of
\cite[arXiv version, p.~6]{CMR1998}. Since the total two-replica
weight of $\{\sigma_v=\tau_v=s\}$ is $Z(s)^2$, $\mathcal N_s\le Z(s)^2$, and (a) follows.
(b) With $s^*=\sgn\xi_G$ the probability is $\mathcal N_{s^*}/Z(s^*)^2\le(Z(-s^*)/Z(s^*))^2=e^{-4|\xi_G|}$.
(c) If $\sigma_v\ne\tau_v$ then $q_u\ne q_v$ and $u\not\leftrightarrow v$, so
$\Pp(u\not\leftrightarrow v)=(2Z(+)Z(-)+2\mathcal N)/(Z(+)+Z(-))^2\le2p_+p_-+2\min(p_\pm)^2=2\min(p_\pm)$,
where $p_\pm=Z(\pm)/(Z(+)+Z(-))$ and $\mathcal N=\mathcal N_\pm$.
\end{proof}

\section{From effective couplings to the bulk}\label{sec:transfer}

This section turns bounds on the laws of the $\xi_k$ into bounds for the bulk.
Lemmas~\ref{lem:S} and~\ref{lem:Bprime} bound, conditionally on $J$, the probability that the
overlap at a vertex $x$ differs from a common overlap at the poles, respectively that $x$ lies in
neither blue pole cluster, by a sum over the ancestor levels $k$ of terms that are small unless the
two halves of the branch leading to $x$ have nearly equal strengths (or, for blue connection, are both
weak); small-ball bounds average these terms (Theorem~\ref{thm:transfer}).

Throughout this section $x$ is a non-pole vertex of level $j$, with the ancestors of
Section~\ref{sec:structure}, and the boundary law is admissible unless stated otherwise. Put
\[
p(s):=\frac1{1+e^{2s}},\qquad \psi(s):=1-(1-p(s))^2,
\]
so that, for a spin in a field of strength $s\ge0$, $p(s)$ is the probability that it takes the
unfavoured sign, $\psi(s)$ the probability that at least one of two independent replicas does, and
$2p(s)(1-p(s))=\frac12\sech^2s$ the probability that the two replicas disagree. Moreover
$\frac12\sech^2s=\psi(s)-p(s)^2\le\psi(s)$.

\subsection{The spin chain}\label{sec:spinchain}

\begin{lemma}[Spin chain]\label{lem:S}
For every $J$,
\[
\Pp(q_x\ne q_A,\ q_A=q_B\mid J)\le\frac12\sum_{k=j}^{N}\sech^2\Delta_k .
\]
\end{lemma}

\begin{proof}
Let $\mathrm{Cons}_k$ be the event that $q$ takes the same value at the two poles of $D_k$. If
$q_A=q_B$ and, for every $k$, $\mathrm{Cons}_k$ implies $q_{y_k}=q_{a_{D_k}}$, then by downward
induction with \eqref{eq:nesting} every $\mathrm{Cons}_k$ holds with the common value $q_A$, and at
$k=j$ this gives $q_x=q_{y_j}=q_A$. Hence
\[
\{q_x\ne q_A,\ q_A=q_B\}\subset\bigcup_{k=j}^{N}\big(\mathrm{Cons}_k\cap\{q_{y_k}\ne q_{a_{D_k}}\}\big).
\]
$\mathrm{Cons}_k$ is a function of the spins at the poles of $D_k$. Condition on these spins. By
Lemma~\ref{lem:markov}(b), $\sigma_{y_k}$ has law $\propto e^{h\sigma}$ with
$h=\xi^{(1)}_k\sigma_{a}+\xi^{(2)}_k\sigma_{b}$ ($a,b$ the poles of $D_k$), and $\tau_{y_k}$ is
independent with field $\xi^{(1)}_k\tau_a+\xi^{(2)}_k\tau_b$. On $\mathrm{Cons}_k$ we have
$\tau_a=q\sigma_a$, $\tau_b=q\sigma_b$ for the common value $q$, so $q\tau_{y_k}$ is an independent
copy of $\sigma_{y_k}$, and
\[
\Pp(q_{y_k}\ne q\mid\cdot)=\Pp(\sigma_{y_k}\ne q\tau_{y_k}\mid\cdot)
=\frac12\sech^2h\le\frac12\sech^2\Delta_k,
\]
because $|h|\ge\big||\xi^{(1)}_k|-|\xi^{(2)}_k|\big|$. A union bound over $k$ completes the proof.
\end{proof}

\begin{remark}[Sharpness of the spin chain]\label{rem:S-sharp}
The constant is attained: if $x$ has level $N$ and the poles are fixed with
$s_As_B=-\sgn(\xi^{(1)}_N\xi^{(2)}_N)$, then $q_A=q_B=1$, $|h|=\Delta_N$, and
$\Pp(q_x\ne1\mid J)=\frac12\sech^2\Delta_N$ exactly. The chain of \cite[\S VII]{CE1984} follows the
same ancestors with fixed poles and controls $\E\langle\sigma_x\rangle^2$ (their
$\E\langle s_0\rangle^2$) through an exact linear
recursion \cite[(7.1)--(7.3)]{CE1984}, a one-step loss off the event that $|\xi|$, $|\xi'|$ or
$|\xi\pm\xi'|$ is small \cite[Lemma~7.8]{CE1984}, and a symmetry argument
\cite[Lemma~7.5]{CE1984}. Lemma~\ref{lem:S} is a two-replica union-bound form of that argument; it
needs no symmetry of the coupling law and applies to every admissible boundary law.
\end{remark}

\subsection{The blue chain}\label{sec:bluechain}

\begin{lemma}[Blue chain]\label{lem:Bprime}
For every $J$,
\[
\Pp\big(x\notin C(A)\cup C(B)\mid J\big)\le\sum_{k=j}^{N}\big[\psi(\Delta_k)+e^{-4\Xi_k}\big].
\]
\end{lemma}

\begin{proof}
Let $\mathrm{Att}_k$ be the event that $y_k$ is joined to $\pi_k$ by blue edges of $H^*_k$. On
$\bigcap_k\mathrm{Att}_k$, downward induction with \eqref{eq:nesting} puts the poles of every $D_k$
in $C(A)\cup C(B)$, and at $k=j$ it puts $x=y_j$ there. Condition on the spins at the poles of $D_k$.
In each replica the log-odds that $(y_k,\pi_k)$ is in the configuration preferred by $H^*_k$,
$\sigma_{y_k}\sigma_{\pi_k}=\sgn\xi_{H^*_k}$, is at least $2(|\xi_{H^*_k}|-|\xi_{\rm other}|)=2\Delta_k$
whatever the other pole does (Lemma~\ref{lem:markov}(b)); the replicas are independent, so both are
preferred with probability at least $(1-p(\Delta_k))^2$. Conditionally also on the spins at $y_k$,
the interior of $H^*_k$ has the fixed-pole two-replica law (Lemma~\ref{lem:markov}(b)), and on the
preferred configuration Lemma~\ref{lem:X}(b) gives
$\Pp(y_k\not\leftrightarrow\pi_k\text{ in }H^*_k\mid\cdot)\le e^{-4|\xi_{H^*_k}|}=e^{-4\Xi_k}$. Hence
$\Pp(\mathrm{Att}_k^c\mid\cdot)\le1-(1-p(\Delta_k))^2+e^{-4\Xi_k}$, and a union bound over $k$ completes the
proof. No condition on the overlaps at the poles is needed: when both replicas are preferred,
$q_{y_k}=q_{\pi_k}$ automatically.
\end{proof}

\subsection{Small-ball profiles}\label{sec:profiles}

The per-level terms of Lemmas~\ref{lem:S} and~\ref{lem:Bprime} are small unless $\Delta_k$ is
small (the halves have nearly equal strengths) or $\Xi_k$ is small (both halves are weak).
Averaged over the couplings, the first is controlled by a bound on the concentration function
$\sup_y\Pp(||\xi_{k-1}|-y|<t)$ and the second by a small-ball bound on $\Pp(|\xi_{k-1}|<t)$; the
rates $\as_k,\ab_k$ below bound these averages, and $S_N,P_N$ the corresponding pole terms
(Lemma~\ref{lem:rates}).

\begin{definition}\label{def:profile}
A \emph{small-ball profile} is a family of nondecreasing functions
$F_k,Q_k:[0,\infty)\to[0,\infty)$, $k\ge0$, such that for all $t\ge0$
\[
\Pp(|\xi_k|<t)\le F_k(t),\qquad \sup_{y\ge0}\Pp\big(\big||\xi_k|-y\big|<t\big)\le Q_k(t).
\]
For $k\ge1$ and $N\ge1$ put
\begin{align*}
\as_k&:=\int_0^\infty\sech^2s\,\tanh s\,\min\{1,Q_{k-1}(s)\}\,ds,\\
\ab_k&:=\int_0^\infty(-\psi'(s))\min\{1,Q_{k-1}(s)\}\,ds+\int_0^\infty4e^{-4s}\min\{1,F_{k-1}(s)\}^2\,ds,\\
S_N&:=\int_0^\infty2\sech^2t\,\tanh t\,\min\{1,F_N(t)\}\,dt,\qquad
P_N:=\int_0^\infty\sech^2t\,\min\{1,F_N(t)\}\,dt,
\end{align*}
and for $1\le j\le N$ let $\Ts_{N,j}:=\sum_{k=j}^N\as_k$ and $\Tb_{N,j}:=\sum_{k=j}^N\ab_k$; for
$x\in\{A,B\}$, which have no level, read $\Ts_{N,k_0(x)}$ and $\Tb_{N,k_0(x)}$ as $0$.
\end{definition}

The superscripts ${\rm s}$ and ${\rm b}$ refer to the spin chain and the blue chain. The law of
$\xi_k$ does not depend on $N$, so neither do $F_k,Q_k,\as_k,\ab_k$.

\begin{lemma}[Averaged rates]\label{lem:rates}
For $k\ge1$: $\E\frac12\sech^2\Delta_k\le\as_k$, $\E[\psi(\Delta_k)+e^{-4\Xi_k}]\le\ab_k$ and
$\as_k\le\ab_k$. Moreover $\E\sech^2\xi_N\le S_N$ and $\E(1-\tanh|\xi_N|)\le P_N$.
\end{lemma}

The proof (Appendix~\ref{app:proofs}) is a layer-cake computation; the key point is that, by
independence of the two halves, $\Pp(\Delta_k<s)\le\sup_y\Pp(||\xi_{k-1}|-y|<s)$ and
$\Pp(\Xi_k<s)=\Pp(|\xi_{k-1}|<s)^2$.

\subsection{The transfer theorem}\label{sec:transferthm}

\begin{theorem}[Transfer]\label{thm:transfer}
Let $(F_k,Q_k)$ be a small-ball profile. For every $N\ge1$ and every non-pole $x$ of level $j$:
\begin{enumerate}
\item[\textup{1.}] Free poles: $\E\langle\sigma_A\sigma_B\rangle^2=1-\E\sech^2\xi_N\ge1-S_N$.
\item[\textup{2.}] Fixed poles, for any $s_A,s_B$: $\E\langle\sigma_x\rangle^2\ge1-2\Ts_{N,j}$. If the
law of $J_e$ is symmetric, $\E\langle\sigma_x\rangle=0$.
\item[\textup{3.}] Free poles: $\E\langle\sigma_A\sigma_x\rangle^2\ge(1-2\Ts_{N,j}-S_N)_+$, and
\[
\E\langle R_{12}^2\rangle\ge\Big(\frac1{|\V_N|}\sum_{z\in\V_N}\big(1-2\Ts_{N,k_0(z)}-S_N\big)_+\Big)^2.
\]
\item[\textup{4.}] Free poles:
\[
\Pp\big(x\notin C(A)\cup C(B)\big)\le\Tb_{N,j},\qquad\Pp(A\not\leftrightarrow B)\le P_N .
\]
Let $\bar X_N:=|\V_N|^{-1}\sum_{z\in\V_N}\Tb_{N,k_0(z)}$. Then for every $\theta\in(0,1)$,
\begin{gather*}
\Pp\big(B\in C(A),\ |C(A)|>(1-\theta)|\V_N|\big)\ge1-\bar X_N/\theta-P_N,\\
\frac{\E|C(A)|}{|\V_N|}\ge\frac1{|\V_N|}\sum_{z\in\V_N}\big(1-\Tb_{N,k_0(z)}-P_N\big)_+ .
\end{gather*}
\item[\textup{5.}] Free poles, $1\le m\le N$:
\[
\Pp\big(M_N>(2n)^{-m}|\V_N|\big)\le P_N+v_m\Tb_{N,N-m+1},\qquad
\E M_N/|\V_N|\le(2n)^{-m}+P_N+v_m\Tb_{N,N-m+1}.
\]
\item[\textup{6.}] Free poles: $\E\,\Imb_N^2\ge(1-2\bar X_N)_+^2-P_N$ and
$\E\,\Imb_N^2\ge\big((\E|C(A)|-\E M_N)/|\V_N|\big)_+^2$.
\end{enumerate}
\end{theorem}

\begin{proof}
1. Lemma~\ref{lem:rec}(c) and Lemma~\ref{lem:rates}.

2. With fixed poles $q_A=q_B=1$, and by \eqref{eq:twopoint}
$\langle\sigma_x\rangle^2=\E[q_x\mid J]=1-2\Pp(q_x\ne q_A,\,q_A=q_B\mid J)$. Apply
Lemma~\ref{lem:S}, average over $J$ and use Lemma~\ref{lem:rates}. If the law of $J_e$ is symmetric,
reversing the signs of the couplings of the edges at $x$ preserves the law of $J$ and maps
$\langle\sigma_x\rangle$ to $-\langle\sigma_x\rangle$ (cf.\ \cite[Lemma~7.5]{CE1984}).

3. $\{q_x\ne q_A\}\subset\{q_A\ne q_B\}\cup\{q_x\ne q_A,q_A=q_B\}$, and
$2\Pp(q_A\ne q_B)=\E\sech^2\xi_N\le S_N$ (Lemma~\ref{lem:rec}(c)); so
$\E[q_Aq_x]=1-2\Pp(q_x\ne q_A)\ge1-S_N-2\Ts_{N,j}$, while $\E[q_Aq_x]=\E\langle\sigma_A\sigma_x\rangle^2\ge0$.
The same bound holds for $z\in\{A,B\}$ with $\Ts:=0$, since $\E[q_Aq_A]=1$ and
$\E[q_Aq_B]=1-\E\sech^2\xi_N$. Since $q_A^2=1$, Cauchy--Schwarz gives
$\E R_{12}^2\ge(\E[R_{12}q_A])^2=(|\V_N|^{-1}\sum_z\E[q_Aq_z])^2$.

4. The first bound is Lemma~\ref{lem:Bprime} with Lemma~\ref{lem:rates}; the second is
Lemma~\ref{lem:X}(c) for $G=\DL_N$ with Lemma~\ref{lem:rates}. Let
$Y:=|\V_N\setminus(C(A)\cup C(B))|/|\V_N|$, so $\E Y\le\bar X_N$. On
$\{A\leftrightarrow B\}\cap\{Y<\theta\}$ we have $B\in C(A)$ and $|C(A)|>(1-\theta)|\V_N|$, and Markov's
inequality gives the first display. For the second,
$\Pp(z\in C(A))\ge\Pp(z\in C(A)\cup C(B),\,A\leftrightarrow B)\ge1-\Tb_{N,k_0(z)}-P_N$.

5. By Lemma~\ref{lem:struct}(d), on $\{A\leftrightarrow B\}\cap\{W_m\subset C(A)\cup C(B)\}$ every blue
cluster other than $C(A)=C(B)$ avoids $W_m\cup\{A,B\}$ and so has at most
$v_{N-m}\le(2n)^{-m}|\V_N|$ vertices. Every $w\in W_m$ has level at least $N-m+1$, so
$\Pp(w\notin C(A)\cup C(B))\le\Tb_{N,N-m+1}$ (the terms $\ab_k$ are nonnegative). A union bound over the
$v_m$ vertices of $W_m$ gives the first bound, and $M_N\le|\V_N|$ gives the second.

6. On $\{A\leftrightarrow B\}$, $C(A)=C(B)$ has overlap $q_A$ and $(1-Y)|\V_N|$ vertices, and every
cluster of overlap $-q_A$ lies in $\V_N\setminus C(A)$; hence $|\Imb_N|\ge1-2Y$ there. Since
$y\mapsto(1-2y)_+^2$ is convex, nonincreasing and at most $1$,
\[
\E\,\Imb_N^2\ge\E\big[(1-2Y)_+^2\1\{A\leftrightarrow B\}\big]\ge\E(1-2Y)_+^2-P_N\ge(1-2\E Y)_+^2-P_N
\ge(1-2\bar X_N)_+^2-P_N .
\]
For the second bound, $|C_{q_A}|\ge|C(A)|$ and $|C_{-q_A}|\le M_N$ pointwise, so
$|\Imb_N|\ge(|C(A)|-M_N)/|\V_N|$, and $\E\,\Imb_N^2\ge(\E|\Imb_N|)^2$.
\end{proof}

\begin{corollary}[Level-weighted limits]\label{cor:limits}
Let $(F_k,Q_k)$ be a small-ball profile with $\Ts_{\infty,j}:=\sum_{k\ge j}\as_k<\infty$ and
$\Tb_{\infty,j}:=\sum_{k\ge j}\ab_k<\infty$, and put
\[
\kappa^{\rm s}:=\sum_{j\ge1}(2n-1)(2n)^{-j}(1-2\Ts_{\infty,j})_+,\qquad
\kappa^{\rm b}:=\sum_{j\ge1}(2n-1)(2n)^{-j}(1-\Tb_{\infty,j})_+ .
\]
\begin{enumerate}
\item[\textup{(a)}] If $S_N\to0$, then $\liminf_N\E\langle R_{12}^2\rangle\ge(\kappa^{\rm s})^2$.
\item[\textup{(b)}] If $P_N\to0$, then $\liminf_N\E|C(A)|/|\V_N|\ge\kappa^{\rm b}$.
\item[\textup{(c)}] If $P_N\to0$, then $\E M_N/|\V_N|\to0$, so that
$M_N/|\V_N|\to0$ in probability, and $\liminf_N\E\,\Imb_N^2\ge(\kappa^{\rm b})^2$.
\end{enumerate}
Both $\kappa^{\rm s}$ and $\kappa^{\rm b}$ are positive, and
$\kappa^{\rm s}\ge\sum_{j\ge1}(2n-1)(2n)^{-j}(1-2\Tb_{\infty,j})_+$.
\end{corollary}

\begin{proof}
Drop the poles and use \eqref{eq:counts} and $\Ts_{N,j}\le\Ts_{\infty,j}$:
\[
|\V_N|^{-1}\sum_z(1-2\Ts_{N,k_0(z)}-S_N)_+\ge(1-\varrho_N)\sum_{j=1}^{N}(2n-1)(2n)^{-j}(1-2\Ts_{\infty,j}-S_N)_+,
\]
and let $N\to\infty$ by dominated convergence over $j$. Part (b) is the same argument with item 4.
For (c), item 5 with $m$ fixed gives $\limsup_N\E M_N/|\V_N|\le(2n)^{-m}$, since
$\Tb_{N,N-m+1}$ is a sum of $m$ terms $\ab_k$ with $k\ge N-m+1$, and $\ab_k\to0$ because
$\sum_k\ab_k<\infty$; let $m\to\infty$. Then item 6 and
(b) give the imbalance bound. Positivity: $\Ts_{\infty,j},\Tb_{\infty,j}\to0$ as $j\to\infty$, and
$\as_k\le\ab_k$ gives the comparison.
\end{proof}

\begin{remark}[What the coupling-growth input must supply]\label{rem:feeders}
Bounds uniform in $N$ need $\sum_k\as_k<\infty$, $\sum_k\ab_k<\infty$, $\sup_NS_N$ and
$\sup_NP_N$ small; the limits of Corollary~\ref{cor:limits} need in addition $S_N\to0$ for part
(a) and $P_N\to0$ for parts (b) and (c).
Section~\ref{sec:partI} supplies all of these from a density bound on an event
(Lemma~\ref{lem:ratesH}); Section~\ref{sec:partII} supplies them from certified distribution
functions, with $Q_k(t)=F_k(2t)$ by unimodality (Section~\ref{sec:proofB}).
\end{remark}

\subsection{Free and fixed poles}\label{sec:freefixed}

Lemmas~\ref{lem:S} and~\ref{lem:Bprime} condition only on the pole spins of each ancestor, so they
hold for every admissible boundary law. Free poles pay only $2\Pp(q_A\ne q_B)=\E\sech^2\xi_N$ in
item 3; fixed poles have $q_A=q_B=1$, and Lemma~\ref{lem:S} gives
$\langle\sigma_x\rangle^2\ge1-\sum_{k=j}^N\sech^2\Delta_k$ pointwise in $J$. Collet and Eckmann
observe that ``the value of the Edwards-Anderson parameter is independent of the boundary
conditions'' \cite[p.~397]{CE1984}. This common treatment of free and fixed poles covers only the
spin items 2 and 3: items 4--6 use $\Pp(A\not\leftrightarrow B)\le P_N$, which holds for free poles (and,
since $e^{-4x}\le1-\tanh x$ for $x\ge0$, for fixed poles in the preferred configuration); for
fixed poles of the other sign we have no such bound.

\section{Coupling growth for \texorpdfstring{$n\ge4$}{n ≥ 4}}\label{sec:partI}

This section proves Theorem~\ref{thm:A}. We propagate a density bound on an event:
$\Hyp_k(r,w)$ (Definition~\ref{def:H}) says that $\xi_k$, restricted to an event of probability
at least $1-w$, has density at most $r$. Proposition~\ref{prop:A} passes from level $k$ to level
$k+1$, multiplying $r$ by a factor $\Gamma$ that is close to $\kappa_n$ when $r,w$ are small and
the cuts $L,R$ are large; Rogozin's inequality (Section~\ref{sec:rogozin}) brings in the constant
$\kappa_n$, and $\kappa_n<1$ exactly when $n\ge4$ (Lemma~\ref{lem:kappa}). Lemmas~\ref{lem:T}
and~\ref{lem:E} iterate the step (Section~\ref{sec:iteration}), and Lemma~\ref{lem:ratesH} turns
the resulting bounds into a small-ball profile for Theorem~\ref{thm:transfer}.

\subsection{Rogozin's inequality and the constants \texorpdfstring{$\kappa_g$}{κ\_g}}\label{sec:rogozin}

For a real random variable $X$ let $\mathsf M(X)$ be the essential supremum of its density
($+\infty$ if it has none). For widths $v_1,\dots,v_g>0$ let
\[
h_v(x):=v^{-1}\1\{|x|\le v/2\},\qquad \Phi_g(v_1,\dots,v_g):=(h_{v_1}*\dots*h_{v_g})(0),
\]
the density at $0$ of a sum of independent centred uniform variables of widths $v_i$. Then
$\Phi_g(\lambda v)=\lambda^{-1}\Phi_g(v)$.

\begin{lemma}[Rogozin]\label{lem:rogozin}
If $X_1,\dots,X_g$ are independent with $\mathsf M(X_i)\le m_i<\infty$, then
$\mathsf M(X_1+\dots+X_g)\le\Phi_g(1/m_1,\dots,1/m_g)$.
\end{lemma}

Rogozin \cite[Theorem, p.~53]{Rogozin1987} proves this with $m_i=\mathsf M(X_i)$ (see also
\cite[Thm.~1.4]{MMX2017}); the extension to upper bounds $m_i$ uses that $\Phi_g$ is nonincreasing
in each width, which is shown in Appendix~\ref{app:proofs}.

Isopi and Newman \cite[Lemma~4]{IN1992} combine Rogozin's inequality with Ball's cube-slicing bound
to control the density of weighted sums, and von Soosten and Warzel \cite{vSW2017} follow density
suprema along a hierarchical renormalization map, with bounds that allow them to grow from level to
level. Here the inequality gives a contraction of the map \eqref{eq:intro-rec}.

\begin{lemma}[The constants $\kappa_g$]\label{lem:kappa}
Let $p_2(x):=2(1-2|x|)_+$ and $\kappa_g:=p_2^{*g}(0)$. Then $\kappa_g=\E\,\Phi_g(V_1,\dots,V_g)$
for independent $V_i$ with density $2v$ on $[0,1]$; $\kappa_g=2f_{2g}(g)$, where $f_m$ is the
density of the sum of $m$ independent uniform variables on $[0,1]$;
$\kappa_1=2$, $\kappa_2=4/3$, $\kappa_3=11/10$, $\kappa_4=302/315$; $\kappa_{g+1}\le\kappa_g$; and
$\kappa_g<1$ if and only if $g\ge4$.
\end{lemma}

\begin{proof}
$\E h_V(x)=\int_0^1 2v\cdot v^{-1}\1\{|x|\le v/2\}\,dv=p_2(x)$, and Tonelli gives
$\E\Phi_g(V)=p_2^{*g}(0)$. Since $p_2$ is the density of $(U_1+U_2-1)/2$ with $U_i$ uniform on
$[0,1]$, $p_2^{*g}$ is the density of $(U_1+\dots+U_{2g}-g)/2$, whose value at $0$ is $2f_{2g}(g)$;
the Irwin--Hall formula
$f_m(x)=\frac1{(m-1)!}\sum_{i=0}^{\lfloor x\rfloor}(-1)^i\binom mi(x-i)^{m-1}$ gives the four
values. The density $p_2^{*g}$ is symmetric unimodal, so
$\kappa_g=\|p_2^{*g}\|_\infty$ and $\kappa_{g+1}\le\|p_2^{*g}\|_\infty\|p_2\|_1=\kappa_g$.
\end{proof}

The law with density $2v$ on $[0,1]$ is that of the larger of two independent uniform variables;
this is how $\kappa_g$ arises in Proposition~\ref{prop:A}, where the probability of the range left
to the smaller half plays the role of $V_i$ (Section~\ref{sec:mechanism}).

\subsection{Density bounds on events}\label{sec:H}

\begin{definition}\label{def:H}
For $r\ge0$ and $w\ge0$, $\Hyp_k(r,w)$ means: there is an event $E_k$, measurable with respect to
the couplings of a scale-$k$ sub-diamond, given by a fixed measurable function of its coupling
vector (so that it is defined on every copy), with $\Pp(E_k^c)\le w$ and
$\Pp(\xi_k\in dx,\,E_k)\le r\,dx$.
\end{definition}

Under (D), $\Hyp_0(\phi/\beta,0)$ holds with $E_0$ the sure event, and $\Hyp_k(r,w)$ implies
$\Hyp_k(r',w')$ for $r'\ge r$, $w'\ge w$.

\begin{lemma}\label{lem:Hcons}
If $\Hyp_k(r,w)$, then $\Pp(|\xi_k|<t)\le w+2rt$ and
$\sup_{y\ge0}\Pp(||\xi_k|-y|<t)\le w+4rt$ for all $t\ge0$.
\end{lemma}

\begin{proof}
$\Pp(\xi_k\in S)\le w+r\,|S|$ for every Borel set $S$, and the sets $\{|x|<t\}$ and
$\{x:||x|-y|<t\}$ have Lebesgue measure at most $2t$ and $4t$.
\end{proof}

\subsection{The one-step contraction}\label{sec:propA}

In the next proposition a branch of a scale-$(k+1)$ sub-diamond is called \emph{good} if both
halves lie in their events, the larger modulus is at least $R$, and it exceeds the smaller by at
least $L$.

\begin{proposition}[One-step contraction]\label{prop:A}
Assume $\Hyp_k(r,w)$ with $w<1$. For $L,R>0$ put
\[
\rho:=\frac r{1-w},\quad \varpi:=2w+(2\rho R)^2+4\rho L,\quad \vartheta:=\frac{1+e^{-2L}}{\tanh R},\quad
\Gamma:=\vartheta\sum_{g=1}^{n}\binom ng\kappa_g\,\varpi^{\,n-g}.
\]
Then $\Hyp_{k+1}(r\Gamma,\,\varpi^n)$ holds.
\end{proposition}

Here $\varpi$ bounds the probability that a branch is not good, and $\vartheta\ge1$ bounds the factor by which the series map can increase the density on a good branch at positive temperature; $\vartheta\to1$ as
$L,R\to\infty$.

\begin{proof}
\emph{Setup.} Let $p_E:=\Pp(E_k)\ge1-w$ and $\nu:=$ the law of $\xi_k$ given $E_k$. Then $\nu$ has
density at most $r/p_E\le\rho$. Let $F(t):=\nu(|\xi|<t)$; $F$ is continuous, and under $\nu$ the
variable $F(|\xi|)$ is uniform on $[0,1]$. Branch $i$ of the scale-$(k+1)$ sub-diamond has halves
with couplings and events $(a_i,E^{a_i})$, $(b_i,E^{b_i})$; the $2n$ pairs are independent and
distributed as $(\xi_k,\1_{E_k})$. Let $\mathcal T_i:=E^{a_i}\cap E^{b_i}$, so $\Pp(\mathcal T_i)=p_E^2$
and $(a_i,b_i)\sim\nu\otimes\nu$ given $\mathcal T_i$. Let
\begin{gather*}
\mathcal G_i:=\{\max(|a_i|,|b_i|)\ge R\}\cap\{||a_i|-|b_i||\ge L\},\qquad
\mathcal B_i:=\mathcal T_i\cap\mathcal G_i,\\
I:=\{i:\mathcal B_i\text{ occurs}\},\qquad E_{k+1}:=\{I\ne\emptyset\}.
\end{gather*}
Branch $i$ is \emph{good} on $\mathcal B_i$. The event $E_{k+1}$ is measurable with respect to the
couplings of the scale-$(k+1)$ sub-diamond.

\emph{(1) Failure.} $\Pp(\mathcal B_i^c)\le(1-p_E^2)+\nu^{\otimes2}(\max(|a|,|b|)<R)+\nu^{\otimes2}(||a|-|b||<L)
\le2w+F(R)^2+4\rho L\le\nobreak\varpi$, using $F(R)\le2\rho R$ and, conditionally on $b$,
$\nu(||a|-|b||<L)\le4\rho L$. By independence, $\Pp(E_{k+1}^c)\le \varpi^n$.

\emph{(2) Conditional densities.} On $\mathcal G_i$ exactly one half is larger in modulus by at least
$L$; call its modulus $c_i\ge R$ and the other half $z_i$. For $|y|\ge R$,
\[
\Pp(\mathcal T_i,\ b_i\in dy,\ a_i\in dx',\ |x'|\le|y|-L)=p_E^2\,\nu(dy)\,\nu(dx')\,\1\{|x'|\le|y|-L\},
\]
so given $\mathcal T_i$, $\{b_i$ larger$\}$ and $b_i=y$, the law of $z_i=a_i$ is $\nu$ restricted to
$[-(|y|-L),|y|-L]$ and normalized by $W_i:=F(c_i-L)$, which is positive almost surely on
$\mathcal B_i$; its density is at most $(r/p_E)/W_i$. The same holds with $a,b$ exchanged. Now
$Y_i:=a_i\star b_i=\sgn(y)\,\chi_{c_i}(z_i)$ with $\chi_c(x):=\artanh(\tanh x\tanh c)$, which is odd,
increasing, and satisfies
\[
\chi_c'(x)=\frac{\tanh c}{1+\sinh^2x/\cosh^2c}\ge\frac{\tanh R}{1+e^{-2L}}=\frac1\vartheta
\qquad(|x|\le c-L,\ c\ge R),
\]
since $\sinh(c-L)/\cosh c\le e^{-L}$. Hence, given $\mathcal T_i$ and the label and value of the
larger half, $Y_i$ has density at most $\vartheta(r/p_E)/W_i$.
Let $\mathcal S$ be the $\sigma$-algebra generated by, for each branch, the indicators of $E^{a_i}$,
$E^{b_i}$, $\mathcal B_i$, on $\mathcal B_i$ the label and the value of the larger half, and off
$\mathcal B_i$ the pair $(a_i,b_i)$. Given $\mathcal S$ the $Y_i$ are independent, because the branches
are independent and $\mathcal S$ is generated by functions of single branches; on $\{I=I_0\}$ with
$|I_0|=g\ge1$, the $Y_i$ with $i\notin I_0$ are $\mathcal S$-measurable and each $Y_i$ with $i\in I_0$
has conditional density at most $\vartheta(r/p_E)/W_i$. By Lemma~\ref{lem:rogozin}, applied to the regular
conditional law, and scaling,
\[
\Pp(\xi_{k+1}\in dx\mid\mathcal S)\le\vartheta\,\frac r{p_E}\,\Phi_g\big((W_i)_{i\in I_0}\big)\,dx
\qquad\text{on }\{I=I_0\}.
\]

\emph{(3) Averaging: the cancellation.} Since $\Phi_g(W)=(h_{W_1}*\dots*h_{W_g})(0)$ and the pairs
$(W_i,\1_{\mathcal B_i})$ are independent across branches, Tonelli gives
$\E[\Phi_g((W_i)_{i\in I_0})\prod_{i\in I_0}\1_{\mathcal B_i}]=(m^{*g})(0)$, where
\begin{align*}
m(x):=\E[h_{W}(x)\1_{\mathcal B}]
&=2p_E^2\!\int\!\nu(dy)\,\1\{|y|\ge R\}\,F(|y|-L)\cdot\frac{\1\{|x|\le F(|y|-L)/2\}}{F(|y|-L)}\\
&=2p_E^2\,\nu\big(|y|\ge R,\ F(|y|-L)>0,\ F(|y|-L)\ge2|x|\big),
\end{align*}
with the convention that the integrand vanishes where $F(|y|-L)=0$.
The factor $F(|y|-L)$, the probability of the range of the smaller half, cancels the density factor
$1/W$. Since $F(|y|-L)\le F(|y|)$ and $F(|y|)$ is uniform under $\nu$,
$m(x)\le2p_E^2\,\nu(F(|y|)\ge2|x|)=p_E^2p_2(x)$, and therefore $(m^{*g})(0)\le p_E^{2g}\kappa_g$.

\emph{(4) Collecting.} Since $E_{k+1}$ and each $\{I=I_0\}$ are $\mathcal S$-measurable, (2), (3),
(1) and independence of the branches give
\begin{align*}
\Pp(\xi_{k+1}\in dx,\,E_{k+1})
&=\sum_{\emptyset\ne I_0\subset\{1,\dots,n\}}\E\big[\1\{I=I_0\}\,\Pp(\xi_{k+1}\in dx\mid\mathcal S)\big]\\
&\le\sum_{I_0\ne\emptyset}\vartheta\,\frac r{p_E}\,\E\Big[\Phi_{|I_0|}\big((W_i)_{i\in I_0}\big)\prod_{i\in I_0}\1_{\mathcal B_i}\Big]
\prod_{j\notin I_0}\Pp(\mathcal B_j^c)\,dx\\
&\le\sum_{g=1}^n\binom ng\,\vartheta\,\frac r{p_E}\,p_E^{2g}\kappa_g\,\varpi^{\,n-g}\,dx
\le r\,\vartheta\sum_{g=1}^n\binom ng\kappa_g\,\varpi^{\,n-g}\,dx=r\Gamma\,dx,
\end{align*}
because $p_E^{2g-1}\le1$ for $g\ge1$. With (1) this is $\Hyp_{k+1}(r\Gamma,\varpi^n)$.
\end{proof}

\begin{remark}[Sharpness and the case $n=3$]\label{rem:sharp}
At zero temperature the recursion becomes $Y=\sgn(a)\sgn(b)\min(|a|,|b|)$. If $a,b$ are uniform on
$[-\frac12,\frac12]$, then $Y$ has density exactly $p_2$, so for $n=3$ the sum of three copies has
density $\kappa_3=11/10$ at $0$, while the input density is $1$. Hence for $n=3$ no one-step bound of
the form $r\mapsto c\,r$ with $c<1$ holds at zero temperature, and, by scaling, the same
obstruction is expected at finite $\beta$ as $r\to0$. For $n\ge4$, $\Gamma\to\vartheta\kappa_n$ as
$r,w\to0$ with $L,R$ fixed, and $\vartheta\kappa_n<1$ once $L,R$ are large.
\end{remark}

\begin{remark}[The role of the event]\label{rem:event}
The series map can send bounded densities to unbounded ones: if $X,Y$ are independent and uniform
on $[-a,a]$ and $[-b,b]$, the formula in the proof of Lemma~\ref{lem:U} (Appendix~\ref{app:proofs})
shows that the density of $|X\star Y|$ behaves like $\log(1/z)/(ab)$ as $z\downarrow0$. Moreover, the
conditioning argument of Proposition~\ref{prop:A} cannot be run without a cut and an event at
positive temperature: if $\xi$ has density
$f$ and $F(c)=\Pp(|\xi|<c)$, then given the larger half $c$ the density of $\chi_c(z)$ at $0$ is
$f(0)/(F(c)\tanh c)$, and its average over $c$, with weight $2f_{|\xi|}(c)F(c)$, equals
$2f(0)\int_0^\infty f_{|\xi|}(c)/\tanh c\,dc=\infty$ when $f$ is continuous at $0$ and $f(0)>0$
(for instance for Gaussian couplings). The cut at $R$ and the event $E_k$ in
Proposition~\ref{prop:A} remove this divergence; the density bound of Proposition~\ref{prop:A} is a
bound for the sum over the $n$ branches, restricted to $E_{k+1}$.
\end{remark}

\subsection{Iteration}\label{sec:iteration}

The tail keeps the failure bound quadratic in the density bound, $w_k\le Ar_k^2$: then
$\varpi=O(r_k)$, so $\Gamma$ stays close to $\vartheta\kappa_n$, and $\varpi^n=O(r_k^n)$ reproduces
the quadratic bound at the next level.

\begin{lemma}[Tail]\label{lem:T}
Suppose $\Hyp_K(r_K,w_K)$ holds and there are $A,\gamma,L,R>0$ with
\textup{(T0)} $w_K\le Ar_K^2$ and $Ar_K^2<1$, \textup{(T1)} $\Gamma(r_K,Ar_K^2)\le\gamma<1$, and
\textup{(T2)} $\varpi(r_K,Ar_K^2)^n\le A\gamma^2r_K^2$, where $\Gamma(r,w)$ and $\varpi(r,w)$ are the quantities
of Proposition~\ref{prop:A} at these $L,R$. Then
$\Hyp_k(r_K\gamma^{k-K},\,Ar_K^2\gamma^{2(k-K)})$ for all $k\ge K$.
\end{lemma}

\begin{proof}
For $0<r\le r_K$, $\varpi(r,Ar^2)/r=2Ar+4R^2r/(1-Ar^2)^2+4L/(1-Ar^2)$ is nondecreasing, so $\varpi(r,Ar^2)$ and
$\Gamma(r,Ar^2)$, a polynomial in $\varpi$ with nonnegative coefficients, are nondecreasing, and
$\varpi(r,Ar^2)^n/r^2=r^{n-2}(\varpi/r)^n$ is nondecreasing for $n\ge2$. Induction with $r_k:=r_K\gamma^{k-K}$:
from $\Hyp_k(r_k,Ar_k^2)$, Proposition~\ref{prop:A} gives
$\Hyp_{k+1}(r_k\Gamma(r_k,Ar_k^2),\varpi(r_k,Ar_k^2)^n)$, and $\Gamma(r_k,Ar_k^2)\le\gamma$,
$\varpi(r_k,Ar_k^2)^n\le r_k^2A\gamma^2=Ar_{k+1}^2$. The case $k=K$ uses monotonicity of $\Hyp$.
\end{proof}

\begin{lemma}[Existence]\label{lem:E}
Let $n\ge4$ and assume \textup{(D)}. There are $r^*>0$, $\gamma\in(0,1)$ and $L,R>0$
such that for every $r_0\in(0,r^*]$ the numerical conditions
\textup{(T0)--(T2)} of Lemma~\ref{lem:T} hold at $K=0$ with $w_0=0$ and $A=1$.
Consequently, $\Hyp_k(r_0\gamma^k,r_0^2\gamma^{2k})$ for all $k\ge0$
whenever $\phi/\beta\le r_0\le r^*$.
\end{lemma}

\begin{proof}
By Lemma~\ref{lem:kappa}, $\kappa_n\le\kappa_4=302/315<1$. Fix $L,R>0$ with $\vartheta\kappa_n<1$
(possible since $\vartheta\to1$ as $L,R\to\infty$) and put $\gamma:=(1+\vartheta\kappa_n)/2$. As $r\downarrow0$,
$\varpi(r,r^2)\to0$ and $\Gamma(r,r^2)\to\vartheta\kappa_n<\gamma$. Choose $r_1\in(0,1)$ with
$\Gamma(r,r^2)\le\gamma$ for $r\le r_1$ (by the monotonicity in the proof of Lemma~\ref{lem:T}), and let
$C:=\varpi(r_1,r_1^2)/r_1$. For $r\le r_1$, $\varpi(r,r^2)^n\le C^nr^n\le\gamma^2r^2$ once
$r^{n-2}\le\gamma^2/C^n$ (here $n\ge3$). Put $r^*:=\min\{r_1,(\gamma^2/C^n)^{1/(n-2)}\}<1$. For
$r_0\le r^*$, (T0) holds since $w_0=0\le r_0^2<1$, and (T1), (T2) hold. If moreover
$\phi/\beta\le r_0$, then $\Hyp_0(\phi/\beta,0)$ gives $\Hyp_0(r_0,0)$ by monotonicity of $\Hyp$,
and Lemma~\ref{lem:T} gives the conclusion.
\end{proof}

\begin{definition}\label{def:certA}
A \emph{certificate at $\beta$} consists of an integer $K\ge0$, numbers $r_0\ge\phi/\beta$ and
$w_0=0$, parameters $(L_k,R_k)_{0\le k<K}$ and numbers $(r_k,w_k)_{1\le k\le K}$ with $w_k<1$,
$r_{k+1}\ge r_k\Gamma(r_k,w_k;L_k,R_k)$ and $w_{k+1}\ge\varpi(r_k,w_k;L_k,R_k)^n$ for $k<K$, and tail
parameters $(A,\gamma,L,R)$ satisfying (T0)--(T2) at $K$. Its sequence $(r_k,w_k)_{k\ge0}$ is
continued by $r_k:=r_K\gamma^{k-K}$ and $w_k:=Ar_K^2\gamma^{2(k-K)}$ for $k\ge K$; by (T0) this can
only increase $w_K$, and for $K=0$ it gives $w_0=Ar_0^2$. By Proposition~\ref{prop:A},
monotonicity of $\Hyp$, and Lemma~\ref{lem:T}, a certificate gives $\Hyp_k(r_k,w_k)$ for every $k$.
Since $\Hyp_0(\phi/\beta,0)$ implies $\Hyp_0(r_0,0)$ whenever $r_0\ge\phi/\beta$, a certificate at
$\beta_1$ is a certificate at every $\beta\ge\beta_1$.
\end{definition}

\subsection{Constants and the proof of Theorem~\ref*{thm:A}}\label{sec:proofA}

\begin{lemma}[Rates under $\Hyp$]\label{lem:ratesH}
If $\Hyp_{k}(r_k,w_k)$ holds for all $k$, then $F_k(t):=w_k+2r_kt$, $Q_k(t):=w_k+4r_kt$ is a
small-ball profile, and
\[
\as_k\le\tfrac12w_{k-1}+2r_{k-1},\qquad \ab_k\le\bar a^{\rm b}(r_{k-1},w_{k-1}),\qquad
S_N\le w_N+2r_N,\qquad P_N\le w_N+2\ln2\,r_N,
\]
with $\bar a^{\rm b}$ as in \eqref{eq:abarb}.
\end{lemma}

The proof combines Lemma~\ref{lem:Hcons} with elementary integrals (Appendix~\ref{app:proofs}).

For a certificate put
\begin{equation}\label{eq:consts}
\begin{aligned}
\ell_{\rm s}&:=\sum_{k\ge0}\big(\tfrac12w_k+2r_k\big),&\qquad
\ell_{\rm b}&:=\sum_{k\ge0}\bar a^{\rm b}(r_k,w_k),\\
S&:=\sup_{N\ge0}(w_N+2r_N),&\qquad P&:=\sup_{N\ge0}(w_N+2\ln2\,r_N).
\end{aligned}
\end{equation}
The tails are geometric, so these are finite and are evaluated in closed form beyond $K$.

\begin{proof}[Proof of Theorem~\ref{thm:A}]
Let $r^*,\gamma,L,R$ be as in Lemma~\ref{lem:E}. For $\beta\ge\phi/r^*$ the sequence $r_k=r_0\gamma^k$, $w_k=r_0^2\gamma^{2k}$ with $r_0=\phi/\beta$ is the sequence of a certificate with $K=0$ and tail parameters $(1,\gamma,L,R)$, and \eqref{eq:consts} gives
\[
\ell_{\rm s}=\frac{2r_0}{1-\gamma}+\frac{r_0^2}{2(1-\gamma^2)},\quad S=2r_0+r_0^2,\quad P=2\ln2\,r_0+r_0^2,
\]
and $\ell_{\rm b}=(1+2\ln2)r_0/(1-\gamma)+O(r_0^2)$. All four are increasing functions of
$r_0=\phi/\beta$ and are $O(1/\beta)$. Choose $\beta_0\ge\phi/r^*$ such that $2\ell_{\rm s}+S<1$ and
$2\ell_{\rm b}+\sqrt P<1$ at $\beta=\beta_0$; then both hold for every $\beta\ge\beta_0$, since both
left sides are increasing in $r_0$.

Now let any certificate be given. Lemma~\ref{lem:ratesH} gives a small-ball profile with
$\Ts_{N,j}\le\ell_{\rm s}$, $\Tb_{N,j}\le\ell_{\rm b}$, $S_N\le S$, $P_N\le P$, and
Theorem~\ref{thm:transfer} gives (i)--(vi): (i) from item 1, with $S_N\le w_N+2r_N\to0$;
(ii) from item 2; (iii) from item 3, since every summand is at least $1-2\ell_{\rm s}-S$; (iv) from
item 4 with $\bar X_N\le\ell_{\rm b}$; (v) from item 5 and Lemma~\ref{lem:ratesH}, and
$\epsilon_{N,m}\to0$ because $r_k,w_k\to0$; (vi) from item 6. In (iv) and (vi), if
$2\ell_{\rm b}+\sqrt P<1$ (as for the certificate above), then $\ell_{\rm b}<\frac12$,
$(1-2\ell_{\rm b})_+^2-P>0$ and $1-P>4\ell_{\rm b}(1-\ell_{\rm b})>\ell_{\rm b}$ (if
$\ell_{\rm b}>0$), so the interval $(\ell_{\rm b}/(1-P),1)$ is nonempty.
\end{proof}

Certificates for standard Gaussian couplings (Tables~\ref{tab:thresholds} and~\ref{tab:oneshot}) and
their verification are described in Appendix~\ref{sec:gauss-certs}.

\section{Coupling growth near the critical temperature}\label{sec:partII}

In this section the couplings are standard Gaussian, so $\xi_0=\beta J_e$ is symmetric with a
nonincreasing density on $(0,\infty)$. The argument has two parts. On symmetric unimodal laws the
exact one-level map on the laws of $|\xi_k|$ preserves symmetric unimodality and is monotone in
the peakedness order (Sections~\ref{sec:su}--\ref{sec:lemM}). Hence a finite chain of explicit
laws, each more peaked than the image of its predecessor, bounds $\Pp(|\xi_k|\le t)$ for finitely
many levels (the bridge, Section~\ref{sec:bridge}); a scale-covariant family of laws, invariant
under the map in the peakedness order with the scale growing by a factor $\lambda>1$ per level,
takes over from there (Section~\ref{sec:family}). Both reduce to finitely many inequalities
between explicit laws, checked by computer (Section~\ref{sec:computation});
Theorem~\ref{thm:growth} combines them into bounds at every level, and Section~\ref{sec:proofB}
feeds these to Theorem~\ref{thm:transfer}.

\subsection{Symmetric unimodal laws and peakedness}\label{sec:su}

A real random variable $X$ is \emph{symmetric unimodal} (s.u.) if $X\overset d=-X$ and the law of
$|X|$ is an atom at $0$ plus a nonincreasing density on $(0,\infty)$; equivalently,
$t\mapsto\Pp(|X|\le t)$ is concave on $[0,\infty)$. A law $\nu$ on $[0,\infty)$ is an \emph{s.u.\
magnitude law} if its distribution function $F_\nu$ is concave on $[0,\infty)$; its signed version
$eY$ ($Y\sim\nu$, $e$ an independent fair sign) is s.u. By Khintchine's theorem
\cite{Khintchine1938,DJ1988}, $X$ is s.u.\ iff $X\overset d=U\Theta$ with $U$ uniform on
$[-1,1]$ and independent of $\Theta\ge0$. By Wintner's theorem \cite{Wintner1938,DJ1988}, sums of independent s.u.\
variables are s.u.; with Khintchine's representation this reduces to the fact that the convolution
of two centred uniform laws is a symmetric trapezoid.

\emph{Peakedness order.} For laws on $[0,\infty)$ write $\nu\lpk\mu$ (``$\nu$ is more peaked than
$\mu$'') if $F_\nu(t)\ge F_\mu(t)$ for all $t\ge0$, and $|X|\gpk\nu$ if $\Pp(|X|\le t)\le F_\nu(t)$
for all $t$. For s.u.\ magnitude laws define
\[
T_\star(\nu):=\operatorname{law}(Y_1\star Y_2),\quad \minT(\nu):=\operatorname{law}\min(Y_1,Y_2)\quad
(Y_1,Y_2\text{ iid }\nu),\qquad
\nu_1\oplus\nu_2:=\operatorname{law}|e_1X_1+e_2X_2|,
\]
with $X_i\sim\nu_i$ and fair signs $e_i$, all independent, and $R(\nu):=T_\star(\nu)^{\oplus n}$. If
$\xi_k$ is s.u., then $\xi^{(1)}\star\xi^{(2)}=e_1e_2(|\xi^{(1)}|\star|\xi^{(2)}|)$ with $e_1e_2$ a fair
sign independent of the moduli, and \eqref{eq:intro-rec} gives
$\operatorname{law}|\xi_{k+1}|=R(\operatorname{law}|\xi_k|)$.

\subsection{Unimodality is preserved}\label{sec:lemU}

\begin{lemma}\label{lem:U}
If $K_1,K_2$ are independent and s.u., then $K_1\star K_2$ is s.u. Hence every $\xi_k$ is s.u., and
$T_\star$, $\minT$, $\oplus$ and $R$ map s.u.\ magnitude laws to s.u.\ magnitude laws.
\end{lemma}

The proof, an explicit computation via (E2) of the density of $X\star Y$ for independent uniform
$X,Y$, is in Appendix~\ref{app:proofs}.

In the variable $\tanh\xi$ the series map is a product, and the product of an s.u.\ variable with
an independent variable is unimodal \cite[Prop.~3.6]{CT1998}; in the variable $\xi$ used here we did
not find Lemma~\ref{lem:U} stated, and it is elementary.

\subsection{Peakedness comparison}\label{sec:lemM}

\begin{lemma}\label{lem:M}
\begin{enumerate}
\item[\textup{(a)}] If $S$ is s.u.\ and $t\ge0$, then $g(x):=\Pp(|S+x|\le t)$ is even and nonincreasing in
$|x|$.
\item[\textup{(b)}] If $S$ is s.u., $W,W'$ are independent of $S$, and $\Pp(|W|\le t)\le\Pp(|W'|\le t)$ for
all $t$, then $\Pp(|S+W|\le t)\le\Pp(|S+W'|\le t)$ for all $t$.
\item[\textup{(c)}] If $Y_1,\dots,Y_m$ and $Y_1',\dots,Y_m'$ are independent families of s.u.\ variables
with $\Pp(|Y_i|\le t)\le\Pp(|Y_i'|\le t)$ for all $i,t$, then
$\Pp(|\sum Y_i|\le t)\le\Pp(|\sum Y_i'|\le t)$ for all $t$.
\end{enumerate}
\end{lemma}

\begin{proof}
(a) For $S$ uniform on $[-r,r]$, $g(x)=|[-t-x,t-x]\cap[-r,r]|/(2r)$, a window of fixed length
moving away from the centre; for $S=0$, $g=\1\{|x|\le t\}$; mix (Khintchine). (b)
$\Pp(|S+W|\le t)=\E g(|W|)$ by independence and evenness, and $g$ is nonincreasing on $[0,\infty)$.
(c) Replace $Y_1,\dots,Y_m$ one at a time using (b); the remaining sums are s.u.\ by Wintner's theorem.
\end{proof}

For (a) cf.\ Anderson's theorem \cite{Anderson1955}, and for (c) cf.\ Birnbaum's comparison of
peakedness \cite{Birnbaum1948} and its extension by Sherman \cite[Lemma~3]{Sherman1955}. Those results are stated for laws with densities, while the laws used
below have atoms at $0$; the proof above covers them.

\begin{corollary}[Monotone operations]\label{cor:monotone}
For s.u.\ magnitude laws, $\nu'\lpk\nu$ implies $T_\star(\nu')\lpk T_\star(\nu)$,
$\minT(\nu')\lpk\minT(\nu)$ and $R(\nu')\lpk R(\nu)$; and $\nu_1'\lpk\nu_1$ implies
$\nu_1'\oplus\nu_2\lpk\nu_1\oplus\nu_2$. The operation $\oplus$ is associative and commutative.
\end{corollary}

\begin{proof}
$T_\star$ and $\minT$: couple by quantiles, $Y_1'\le Y_1$ and $Y_2'\le Y_2$, and use (E0). $\oplus$:
Lemma~\ref{lem:M}(b) with $S=e_2X_2$. $R$: $T_\star$, then Lemma~\ref{lem:M}(c), whose summands are
s.u.\ by Lemma~\ref{lem:U}.
\end{proof}

\subsection{The bridge}\label{sec:bridge}

\begin{definition}\label{def:bridge}
A \emph{bridge certificate} for $(n,\beta)$ is a finite sequence $\nu_0,\dots,\nu_J$ of laws on
$[0,\infty)$, where the integer $J$ is the number of bridge steps, each law given explicitly, with
distribution functions $F_j:=F_{\nu_j}$, such that
\begin{enumerate}
\item[\textup{(S)}] every $\nu_j$ is an s.u.\ magnitude law;
\item[\textup{(B0)}] $F_0(t)\ge\Pp(\beta|J_e|\le t)=\operatorname{erf}(t/(\beta\sqrt2))$ for all $t\ge0$;
\item[\textup{(B$j$)}] $\nu_{j+1}\lpk R(\nu_j)$ for $0\le j<J$.
\end{enumerate}
\end{definition}

\begin{theorem}[Bridge]\label{thm:bridge}
Under Definition~\ref{def:bridge}, $\Pp(|\xi_j|\le t)\le F_j(t)$ for all $j\le J$ and $t\ge0$, at every
$\beta'\ge\beta$.
\end{theorem}

\begin{proof}
At $\beta'\ge\beta$, $\Pp(\beta'|J_e|\le t)\le\Pp(\beta|J_e|\le t)\le F_0(t)$, so
$\nu_0\lpk\operatorname{law}|\xi_0|$. If $\nu_j\lpk\operatorname{law}|\xi_j|$, then by (B$j$),
Corollary~\ref{cor:monotone} and Lemma~\ref{lem:U},
$\nu_{j+1}\lpk R(\nu_j)\lpk R(\operatorname{law}|\xi_j|)=\operatorname{law}|\xi_{j+1}|$.
\end{proof}

Condition (B$j$) is checked by computer through a decomposition of the $n$-fold sum into pairwise
sums and coarsenings (replacements of a law by a more peaked s.u.\ law on a coarser grid), which is
justified by Lemma~\ref{lem:decomp} in Appendix~\ref{app:decomp}.

\subsection{The scale-covariant family}\label{sec:family}

For $\alpha\in[0,1]$ and a law $X_c$ on $[0,\infty)$ with distribution function $F_c$, let
$X^{(\alpha)}$ be $0$ with probability $\alpha$ and distributed as $X_c$ otherwise, so that
$\Pp(LX^{(\alpha)}\le t)=\alpha+(1-\alpha)F_c(t/L)$.

The idea of the family is as follows. Suppose that $|\xi_k|$ is less peaked than $LX^{(\alpha)}$,
the shape $X_c$ at scale $L$ with an atom of mass $\alpha$ at $0$. By (E0)--(E1) of
Lemma~\ref{lem:series}, a branch of the next level contributes a term of modulus at least
$L(\min(X,X')-\varepsilon)_+$, where $X,X'$ are independent copies of $X^{(\alpha)}$ and
$\varepsilon=c_0/L$ is the loss of the series map in units of the scale; the branch is
\emph{dead}, and contributes nothing, if one of its inputs comes from the atom, which has
probability at most $2\alpha$. The atom is kept proportional to the loss, $\alpha=A\varepsilon$,
where $A$ is the constant of the family certificate below (not the pole $A$). Without atom and
loss, the sums of $m$ branch terms are controlled by the distribution functions $G_m$ of the
explicit laws $\eta_m$ of (C0), and (C1) says that the sum of $n$ terms is less peaked than
$\lambda X_c$: the scale grows by the factor $\lambda$. Condition (C2) says that the first-order
corrections in $\varepsilon$---the shift by at most $n\varepsilon$, bounded through $\sH_m$ by
concavity, and the dead branches, whose number is binomial---are paid for by the slack $s$ of (C1)
and by the atom $\alpha/\lambda$ of the target law. Condition (C2) only gets easier as
$\varepsilon$ decreases, and $\varepsilon_i=c_0/L_i$ decreases along the levels, so a single check at $\bareps$, an upper bound for the largest loss $\varepsilon_0=c_0/L_0$, covers every level (Theorem~\ref{thm:family}). The shape $X_c$
is fixed; in the computation it is a scaled numerical approximation of a zero-temperature fixed
shape, a heuristic choice that plays no role in the proof.

\begin{definition}\label{def:family}
A \emph{family certificate} for $n$ consists of
\begin{itemize}
\item an atomless s.u.\ magnitude law $X_c$ with piecewise-linear distribution function $F_c$,
$F_c(x_{\max})=1$ for some $x_{\max}$, and $f_c:=F_c'(0+)<\infty$;
\item rationals $\lambda>1$, $A>0$ and $\bareps>0$ with $2A\bareps\le1$;
\item s.u.\ magnitude laws $\eta_1,\dots,\eta_n$ with piecewise-linear distribution functions
$G_1,\dots,G_n$ and right derivatives $g_m:=G_m'(\cdot+)$;
\end{itemize}
such that, with $\sH_m:=G_m+n\bareps g_m$ $(1\le m\le n)$, $\hsH_m:=\max(\sH_m,G_n)$ $(1\le m\le n-1)$
and $\hsH_0:=1$,
\begin{enumerate}
\item[\textup{(C0)}] $\eta_1\lpk\minT(X_c)$ and $\eta_{m+1}\lpk\eta_m\oplus\eta_1$ for $1\le m<n$;
\item[\textup{(C1)}] $s(t):=F_c(t/\lambda)-G_n(t)\ge0$ for $t\in[0,\lambda x_{\max}]$;
\item[\textup{(C2)}] for $t\in[0,\lambda x_{\max})$,
\begin{gather*}
s(t)+\bareps\,\frac A\lambda\big(1-F_c(t/\lambda)\big)\ge\bareps\,\cE(\bareps,t),\\
\cE(\varepsilon,t):=n\,g_n(t)+\sum_{k=1}^{n}\binom nk(2A)^k\varepsilon^{k-1}\big[\hsH_{n-k}(t)-G_n(t)\big],
\end{gather*}
where in $\cE(\varepsilon,t)$ the functions $\sH_m,\hsH_m$ are formed with $\varepsilon$ in place of $\bareps$.
\end{enumerate}
Replacing each $\hsH_m$ by a larger function gives a stronger condition.
\end{definition}

By (C0) and Corollary~\ref{cor:monotone}, $\Pp(|\sum_{r\le m}e_rM_r|\le t)\le G_m(t)$ for independent
$M_r\sim\minT(X_c)$ and fair signs $e_r$. The condition $2A\bareps\le1$, used in the proof of
Theorem~\ref{thm:family}, is listed for convenience; it also follows from (C1) and (C2) at $t=0$: $F_c(0)=0$ because $X_c$ is atomless, so (C1) forces $G_n(0)=0$ and $s(0)=0$; all
brackets are nonnegative, and the term $k=n$ gives $A/\lambda\ge(2A)^n\bareps^{\,n-1}$, that is
$(2A\bareps)^{n-1}\le1/(2\lambda)<1$ (for $n\ge2$).

\begin{theorem}[Invariance of the family]\label{thm:family}
Assume Definition~\ref{def:family} and that $\xi_0$ is s.u. Suppose that at some level $J$
\[
\Pp(|\xi_J|\le t)\le\alpha+(1-\alpha)F_c(t/L)\quad(t\ge0),\qquad L\ge c_0/\bareps,\quad
0\le\alpha\le Ac_0/L .
\]
Put $L_i:=L\lambda^i$, $\varepsilon_i:=c_0/L_i$ and $\alpha_i:=A\varepsilon_i$. Then for all $i\ge0$ and
$u\ge0$
\begin{equation}\label{eq:family}
\Pp(|\xi_{J+i}|\le u)\le\alpha_i+(1-\alpha_i)F_c(u/L_i)\le\frac{Ac_0+f_cu}{L_i}.
\end{equation}
\end{theorem}

\begin{proof}
Note $\varepsilon_i\le\bareps$ and $\alpha_i\le A\bareps\le\frac12$. We show
$|\xi_{J+i}|\gpk L_iX^{(\alpha_i)}$ by induction. For $i=0$: $\alpha\le\alpha_0$ gives
$\alpha+(1-\alpha)F_c\le\alpha_0+(1-\alpha_0)F_c$, so $X^{(\alpha_0)}$ is the more peaked law. Step
$i\to i+1$; write $\varepsilon=\varepsilon_i$, $\alpha=\alpha_i$, $L=L_i$.
\begin{enumerate}
\item $\xi_{J+i+1}=\sum_{r=1}^nY_r$ with $Y_r=\xi^{(r,1)}\star\xi^{(r,2)}$ s.u.\ (Lemma~\ref{lem:U}) and
$|Y_r|=|\xi^{(r,1)}|\star|\xi^{(r,2)}|$.
\item Couple by quantiles, $|\xi^{(r,k)}|\ge LX_{rk}$ with $X_{rk}$ iid $X^{(\alpha)}$. By (E0) and
(E1), $|Y_r|\ge(LX_{r1})\star(LX_{r2})\ge(LM_r-c_0)_+=L(M_r-\varepsilon)_+$ with
$M_r:=\min(X_{r1},X_{r2})$.
\item $Y_r':=e_rL(M_r-\varepsilon)_+$ is s.u.: the distribution function of $(M_r-\varepsilon)_+$ is
$1-(1-F_\alpha(t+\varepsilon))^2$, concave, where $F_\alpha:=\alpha+(1-\alpha)F_c$ is concave.
\item By Lemma~\ref{lem:M}(c), $\Pp(|\xi_{J+i+1}|\le Lu)\le\Pp(|\Upsilon|\le u)$ with
$\Upsilon:=\sum_re_r(M_r-\varepsilon)_+$.
\item \emph{Dead branches.} Call branch $r$ dead if one of its inputs comes from the atom; then
$(M_r-\varepsilon)_+=0$. The number $k$ of dead branches is binomial with parameters $n$ and
$\pi_\alpha:=1-(1-\alpha)^2\le2\alpha$, and given the dead set the other $M_r$ are iid $\minT(X_c)$ and
independent of the signs. Since $|(M-\varepsilon)_+-M|\le\varepsilon$, given a dead set of size $k<n$
\[
\Pp(|\Upsilon|\le u\mid k)\le\Pp\Big(\Big|\sum_{\rm alive}e_rM_r\Big|\le u+n\varepsilon\Big)\le
G_{n-k}(u+n\varepsilon)\le\sH_{n-k}(u),
\]
by (C0) and concavity of $G_{n-k}$, where here and below $\sH_m,\hsH_m$ are formed with
$\varepsilon$ in place of $\bareps$, as in (C2); for $k=n$ the bound is $1=\hsH_0$. With
$\zeta_k:=\Pp(\mathrm{Bin}(n,\pi_\alpha)=k)$ and $\hsH_n:=\sH_n$ this gives
$\Pp(|\Upsilon|\le u)\le\sum_{k=0}^n\zeta_k\hsH_{n-k}(u)$.
\item Every bracket $\hsH_{n-k}-G_n$ is nonnegative ($\hsH_n-G_n=n\varepsilon g_n$), $\zeta_0\le1$ and
$\zeta_k\le\binom nk\pi_\alpha^k\le\binom nk(2A\varepsilon)^k$, so
\[
\sum_k\zeta_k\hsH_{n-k}-G_n=\sum_k\zeta_k(\hsH_{n-k}-G_n)\le n\varepsilon g_n
+\sum_{k\ge1}\binom nk(2A\varepsilon)^k[\hsH_{n-k}-G_n]=\varepsilon\,\cE(\varepsilon,u).
\]
\item \emph{Target.} $\Pp(L_{i+1}X^{(\alpha_{i+1})}\le Lu)=\alpha/\lambda+(1-\alpha/\lambda)F_c(u/\lambda)
=F_c(u/\lambda)+\varepsilon(A/\lambda)(1-F_c(u/\lambda))$. For $u\ge\lambda x_{\max}$ this is $1$. For
$u<\lambda x_{\max}$ it suffices that $s(u)/\varepsilon+(A/\lambda)(1-F_c(u/\lambda))\ge\cE(\varepsilon,u)$. By
(C1) the left side is nonincreasing in $\varepsilon$; $\cE(\varepsilon,u)$ is nondecreasing in $\varepsilon$ (the
functions $\sH_m,\hsH_m$ and the powers $\varepsilon^{k-1}$ increase, and the brackets are nonnegative);
so (C2) at $\bareps$ gives it for every $\varepsilon\le\bareps$.
\end{enumerate}
Hence $|\xi_{J+i+1}|\gpk L_{i+1}X^{(\alpha_{i+1})}$. The last inequality of \eqref{eq:family} uses
$(1-\alpha_i)F_c\le F_c$ and $F_c(v)\le f_cv$ (concavity and $F_c(0)=0$).
\end{proof}

Nothing in the proof depends on the value of $n$. The same architecture---finitely many exact
renormalization steps followed by a region that the map preserves---was used by Okuyama and Ohzeki
\cite{OO2026} to prove the opposite conclusion for $\pm J$ couplings with $n=2$ (and, for larger
scale factors, even $n$ under their sufficient condition). An analytic antecedent of an invariant
family for this recursion, at large $n$, is the alternation of Lemmas~4.2 and~4.3 of
\cite{CE1984}.

\subsection{Certificates and coupling growth}\label{sec:certB}

\begin{definition}\label{def:certB}
A \emph{certificate} for $(n,\beta)$ is a bridge certificate (Definition~\ref{def:bridge}), a family
certificate (Definition~\ref{def:family}) for the same $n$, and the entry condition
\begin{enumerate}
\item[\textup{(EN)}] $F_J(t)\le\alpha_0'+(1-\alpha_0')F_c(t/L_0)$ for all $t\ge0$, with $L_0\ge c_0/\bareps$ and
$0\le\alpha_0'\le Ac_0/L_0$.
\end{enumerate}
\end{definition}

\begin{theorem}[Coupling growth]\label{thm:growth}
If a certificate for $(n,\beta)$ exists, then at every $\beta'\ge\beta$, for all $j\le J$, $i\ge0$ and
$t,u\ge0$, with $L_i=L_0\lambda^i$ and $\alpha_i=Ac_0/L_i$,
\[
\Pp(|\xi_j|\le t)\le F_j(t),\qquad
\Pp(|\xi_{J+i}|\le u)\le\alpha_i+(1-\alpha_i)F_c(u/L_i)\le\frac{Ac_0+f_cu}{L_0\lambda^i}.
\]
Consequently $\E\sech^2\xi_{J+i}\le(Ac_0+f_c)/L_i$, $\E(1-\tanh|\xi_{J+i}|)\le(Ac_0+f_c\ln2)/L_i$, and
$|\xi_N|\to\infty$ in probability.
\end{theorem}

\begin{proof}
Theorem~\ref{thm:bridge}, (EN), and Theorem~\ref{thm:family} with $L=L_0$ and $\alpha=\alpha_0'$; the
two expectations follow from Lemma~\ref{lem:rates} applied to $F(t)=\alpha_i+f_ct/L_i$, using
$\int_0^\infty2t\sech^2t\tanh t\,dt=1$ and $\int_0^\infty t\sech^2t\,dt=\ln2$.
\end{proof}

\subsection{Proofs of Theorem~\ref*{thm:B} and Corollary~\ref*{cor:Bn}}\label{sec:proofB}

Fix one of the five certificates of Table~\ref{tab:points}, verified as described in
Section~\ref{sec:computation}, and $\beta\ge\beta^*$. By Theorem~\ref{thm:growth} and Lemma~\ref{lem:U},
\[
F_k:=F_{\nu_k}\ (k\le J),\qquad F_{J+i}(t):=\alpha_i+f_ct/L_i\ (i\ge1),\qquad Q_k(t):=F_k(2t)
\]
is a small-ball profile: since $|\xi_k|$ is an atom at $0$ plus a nonincreasing density, every interval
of length $2t$ in $[0,\infty)$ has at most the mass of $[0,2t]$, so
$\sup_y\Pp(||\xi_k|-y|<t)\le\Pp(|\xi_k|\le2t)\le F_k(2t)$. For each $l\ge1$ let $\bar\phi_l$ be a
rational upper bound for $\ab_l$ computed from this profile. For $l-1\le J$ it is evaluated by the
checker (check (K) of Appendix~\ref{sec:checks}) and satisfies
\[
\bar\phi_l\ge\int_0^\infty(-\psi'(s))F_{l-1}(2s)\,ds+\int_0^\infty4e^{-4s}F_{l-1}(s)^2\,ds
=\E\psi(V/2)+\E e^{-4\max(V,V')}
\]
with $V,V'$ iid $\nu_{l-1}$; for $l-1=J+i$, $i\ge1$, the closed form
$\ab_l\le a_1/L_i+a_2/L_i^2$ holds with
\[
a_1:=\tfrac34Ac_0+2f_c\big(\tfrac12\ln2+\tfrac14\big),\qquad a_2:=A^2c_0^2+\tfrac12Ac_0f_c+\tfrac18f_c^2,
\]
(here $A$ is the constant of the family certificate), from
\begin{gather*}
\int_0^\infty(-\psi')(\alpha+2fs/L)\,ds=\tfrac34\alpha+(2f/L)\big(\tfrac12\ln2+\tfrac14\big),\\
\int_0^\infty4e^{-4s}(\alpha+fs/L)^2\,ds=\alpha^2+\alpha f/(2L)+f^2/(8L^2),
\end{gather*}
and we put
$\bar\phi_l:=a_1^+/L_i+a_2^+/L_i^2$, where $a_1^+,a_2^+$ are $a_1,a_2$ with $c_0$ and
$\frac12\ln2+\frac14$ replaced by rational upper bounds. Put $\bar T_j:=\sum_{l\ge j}\bar\phi_l$
(finite; the tail is geometric). We use the blue rates for both constants below; the spin chain has
the smaller rates $\as$ (Lemma~\ref{lem:rates}), which would give a larger $\kappa_Q$ but have not
been evaluated rigorously. Put
\begin{equation}\label{eq:kappas}
\kappa_Q:=\sum_{j\ge1}(2n-1)(2n)^{-j}(1-2\bar T_j)_+,\qquad
\kappa_C:=\sum_{j\ge1}(2n-1)(2n)^{-j}(1-\bar T_j)_+ .
\end{equation}
Both are positive because $\bar T_j\to0$.

\begin{proof}[Proof of Theorem~\ref{thm:B}]
Item (i) is Theorem~\ref{thm:growth}, with $\E\langle\sigma_A\sigma_B\rangle^2=1-\E\sech^2\xi_N\to1$
and ${\Pp(A\not\leftrightarrow B)}\le\E(1-\tanh|\xi_N|)\to0$ (Lemma~\ref{lem:X}(c)). The profile above has
$S_N,P_N\to0$ and $\ab_k\to0$. In Corollary~\ref{cor:limits}, $\Tb_{\infty,j}\le\bar T_j$, so
$\kappa^{\rm b}\ge\kappa_C$ and $\kappa^{\rm s}\ge\kappa_Q$, which gives (ii)--(iv).
\end{proof}

\begin{proof}[Proof of Corollary~\ref{cor:Bn}]
Write $\xi^{(m)}_k$ for the effective coupling of the lattice with $m$ branches at the same $\beta$.
We show $\Pp(|\xi^{(n')}_k|\le t)\le\Pp(|\xi^{(n)}_k|\le t)$ for all $k$ and $t\ge0$ by induction on
$k$; for $k=0$ both variables are $\beta J_e$. In the step write $\xi^{(n')}_{k+1}=S+W$, where $S$ is
the sum of the first $n$ branch terms in \eqref{eq:intro-rec} and $W$ the sum of the other $n'-n$.
All terms are independent and s.u.\ (Lemma~\ref{lem:U}), so Lemma~\ref{lem:M}(b) with $W'=0$ gives
$\Pp(|S+W|\le t)\le\Pp(|S|\le t)$. Moreover
$\operatorname{law}|S|=R(\operatorname{law}|\xi^{(n')}_k|)$ with $R$ formed with $n$ branches, and the
induction hypothesis says $\operatorname{law}|\xi^{(n)}_k|\lpk\operatorname{law}|\xi^{(n')}_k|$; by
Corollary~\ref{cor:monotone}, $\operatorname{law}|\xi^{(n)}_{k+1}|=R(\operatorname{law}|\xi^{(n)}_k|)\lpk
\operatorname{law}|S|$, which closes the induction. Hence every upper bound of
Theorem~\ref{thm:growth} for $n$ holds for $n'$, and since Lemma~\ref{lem:U} holds for every number
of branches, the functions $F_k$ and $Q_k$ above form a small-ball profile for the lattice with $n'$
branches. The proof of Theorem~\ref{thm:B} then applies verbatim, with $n'$ in the counting weights
of Corollary~\ref{cor:limits}.
\end{proof}

\section{The computation}\label{sec:computation}

This section states what the computer verifies for the five certificates of
Table~\ref{tab:points} and which proved statements the verification relies on; the certificates
that make Theorem~\ref{thm:A} explicit for standard Gaussian couplings
(Tables~\ref{tab:thresholds} and~\ref{tab:oneshot}; Theorem~\ref{thm:A} itself does not depend on
them) are described in Appendix~\ref{sec:gauss-certs}. The code for both, the frozen inputs, and
the verification logs are available at \url{https://github.com/PeaBrane/mk-spin-glass-order},
commit \href{https://github.com/PeaBrane/mk-spin-glass-order/tree/\repocommit}{\texttt{\repocommit}}.

Computer-assisted proofs in renormalization theory go back to Lanford's proof of the Feigenbaum
conjectures \cite{Lanford1982}, Eckmann, Koch and Wittwer's proof of universality for
area-preserving maps \cite{EKW1984}, whose second part sets up interval arithmetic on a computer,
and Koch and Wittwer's non-Gaussian fixed point for hierarchical scalar models \cite{KW1986}; these
works enclose fixed points of renormalization maps in function spaces. Here the enclosed objects
are distribution functions of random couplings, compared in the peakedness order. The
discretization is one-sided: the bridge laws are staircase laws, and every rounding and every
coarsening replaces a distribution function by a larger one, that is, by a more peaked law, so
that monotonicity of the map preserves the comparison
(Appendices~\ref{app:decomp}--\ref{sec:checks}). This follows the principle of the degrading
quantization of Tal and Vardy \cite{TalVardy2013}: polar bit-channels, whose construction is an
instance of density evolution, are replaced by quantized channels degraded with respect to them,
and degradation is preserved by the polar transforms. Here the order is the peakedness order, and
the series and parallel maps are the check-node and variable-node maps of belief propagation.

For each row of Table~\ref{tab:points} the frozen data are: a family (the shape $X_c$, the
rationals $\lambda$, $A$, $\bareps$, and the witness laws $\eta_1,\dots,\eta_n$), and a bridge (the
exact $\beta$, a dyadic grid, the parameters of the enclosure of the series map $T_\star$, a
schedule of coarsenings (replacements of a law by a more peaked one on a coarser grid), and $J$).
The checker rebuilds the laws $\nu_0,\dots,\nu_J$ deterministically from these data and verifies
(S), (B0), (B$j$) through the decompositions of Lemma~\ref{lem:decomp}, (C0)--(C2), and (EN), each
on the whole half-line. All laws are staircase laws (an atom at $0$ plus nonincreasing masses on
the bins of a uniform grid, in integer units of $2^{-48}$), with concave piecewise-linear
distribution functions, so each condition reduces to finitely many comparisons, bin by bin or at
knots. These are decided in integer and exact rational arithmetic; $\operatorname{erf}$,
$\delta$, $\psi$, $e^{-4s}$ and $\ln2$ are enclosed by interval arithmetic with directed rounding,
and no floating-point comparison decides an inequality (Appendix~\ref{sec:arith}). The checker then
evaluates $\kappa_Q$ and $\kappa_C$ from the certified profile of Section~\ref{sec:proofB}. The
mathematical inputs that are not computed---Lemmas~\ref{lem:series}, \ref{lem:X}, \ref{lem:U},
\ref{lem:M} and~\ref{lem:decomp}, Corollaries~\ref{cor:limits} and~\ref{cor:monotone}, and
Theorems~\ref{thm:transfer}, \ref{thm:bridge}, \ref{thm:family} and~\ref{thm:growth}---are proved
in the text and in Appendices~\ref{app:proofs} and~\ref{app:computation}.

The arithmetic model, the individual checks, the validation runs and the reproduction commands are
given in Appendices~\ref{sec:arith}--\ref{sec:repro}.

\section{Discussion and open problems}\label{sec:discussion}

\emph{An analytic proof for $n=3$.} The one-step constant $\kappa_3=11/10$ is attained
(Remark~\ref{rem:sharp}), so an analytic proof for $n=3$ along the lines of Section~\ref{sec:partI}
would need a multi-level anti-concentration estimate: a bound
$\mathsf M(\xi_{k+m})\le c\,\mathsf M(\xi_k)$ with $c<1$ for some $m\ge2$, on a class of laws that the map
preserves. Theorem~\ref{thm:B} covers $n=3$ only for $T\le4/15$.

\emph{An alternative route for $n\ge4$.}
If $\Pp(|\xi_k|\le u)\le F(u)$ and $|\xi_k|$ has a nonincreasing density on $(0,\infty)$, then
$\Hyp_k(r,w)$ holds with $E_k=\{|\xi_k|\ge u_0\}$, $w=F(u_0)$ and $r=\frac12\inf_{0<v\le u_0}F(v)/v$,
because $v\,f_{|\xi_k|}(|x|)\le\int_0^vf_{|\xi_k|}\le F(v)$ for $|x|\ge u_0\ge v$. For $n\ge4$ a bridge
can therefore be handed to Section~\ref{sec:partI} in place of the family. We do not use this
route; we expect it to need more bridge levels.

\emph{$\pm J$ couplings.} For atomic laws all effective couplings are atomic,
Lemma~\ref{lem:U} fails, and the transfer theorem needs decay of the collision probabilities
$\Pp(||\xi_k|-|\xi_k'||<t)$, summable over levels. For $n=2$ order fails \cite{OO2026}, and for
larger scale factors it fails for those even $n$ that satisfy their sufficient condition; odd $n$
and the even $n\ge4$ of scale factor $2$ are open, as Okuyama and Ohzeki note.

\emph{Closing the gap to $T_c$.} A proof for every $T<T_c(n)$ needs an argument near the critical
fixed point. The constants $\kappa_Q,\kappa_C$ are astronomically small because the transfer uses
a union bound over all levels of the ancestor chain (Lemmas~\ref{lem:S} and~\ref{lem:Bprime}),
dominated by the levels near the critical scale (Remark~\ref{rem:scope}(d)); a product-form bound
along the ancestor chain would improve them, but is not available.

\emph{$\Z^d$.} On $\DL_N$ imbalance takes the degenerate form described in
Section~\ref{sec:cmr-intro}, one giant pole cluster plus sub-extensive dust. On $\Z^d$ the expected
picture, if there is order, is two infinite blue clusters of unequal densities \cite{MNS2008}; in
limits of torus measures at most two infinite blue clusters can exist \cite{Pei2026b}, and blue
percolation is proved in high dimensions, but imbalance is open \cite{Pei2026}.

\emph{General scale factors.} For scale factor $b\ge3$ and $n$ branches we expect the analogue of
Proposition~\ref{prop:A} to have the constant $\kappa^{(b)}_n=p_b^{*n}(0)$,
$p_b(x)=b(1-2|x|)^{b-1}$ on $[-\frac12,\frac12]$ (so $\kappa_g=\kappa^{(2)}_g$), and
$\kappa^{(b)}_n<1$ holds for instance for $(n,b)=(7,3)$, where, by exact evaluation,
$\kappa^{(3)}_7=18905897228749/19864965120000$. The transfer lemmas of Section~\ref{sec:transfer}
are written for $b=2$; we do not pursue general $(n,b)$ here.

\appendix

\section{Proofs of auxiliary lemmas}\label{app:proofs}

This appendix proves, in order of appearance, the routine lemmas of
Sections~\ref{sec:recursion}--\ref{sec:partII}; their statements are in the main text.

\begin{proof}[Proof of Lemma~\ref{lem:struct}]
(a) Vertices created inside $D$ receive edges only inside $D$. (b) Induction on $N$: the
sub-diamonds of $\DL_N$ are $\DL_N$ and the sub-diamonds of its $2n$ halves; a sub-diamond of a
half $H$ other than $H$ has its interior inside $\operatorname{int}H$; distinct halves have
disjoint interiors. (c) The effective couplings are functions of disjoint sets of couplings.
(d) The vertices of level at least $N-m+1$ are the non-pole vertices of the copy of $\DL_m$ in
which every edge has been replaced by a scale-$(N-m)$ sub-diamond, so
$|W_m|=\sum_{i=0}^{m-1}n(2n)^i=v_m$. A vertex $x$ of level at most $N-m$ lies in the interior of
a unique scale-$(N-m)$ sub-diamond, which we denote $D_{N-m}(x)$ (uniqueness by (b)). If $xy$ is
an edge and both endpoints
avoid $W_m\cup\{A,B\}$, then $xy$ is an edge of $D_{N-m}(x)$ by (a) and $y$ is not a pole of it, so
$D_{N-m}(y)=D_{N-m}(x)$. Finally $v_k+c=2n(v_{k-1}+c)$ with $c=n/(2n-1)\in(0,1]$, so
$v_{N-m}\le(2n)^{-m}(v_N+c)\le(2n)^{-m}|\V_N|$.
\end{proof}

\begin{proof}[Proof of Lemma~\ref{lem:rec}]
(a) Flip all interior spins. (b) Series:
$\sum_{m}e^{\xi_1sm+\xi_2ms'}=2\cosh(\xi_1s+\xi_2s')\propto e^{\xi ss'}$ with
$\tanh\xi=\tanh\xi_1\tanh\xi_2$; parallel: weights multiply. (c) The law of $\sigma_A\sigma_B$ is
proportional to $e^{\xi_N\sigma_A\sigma_B}$, and $q_Aq_B=-1$ iff the two replicas disagree on
$\sigma_A\sigma_B$, which has probability $2p(1-p)=\frac12\sech^2\xi_N$ with $p=p(\xi_N)$ as in
Section~\ref{sec:transfer}. (d) Sum out the interiors of the halves and of the other branches.
\end{proof}

\begin{proof}[Proof of Lemma~\ref{lem:series}]
(E0): $\tanh$ is increasing on $[0,\infty)$ with values in $[0,1)$, so $\tanh a\tanh b$ is
symmetric, nondecreasing in each argument and at most $\min(\tanh a,\tanh b)$. (E1) follows from the displayed identity for $e^{2(a\star b)}$ by writing
$\cosh(a+b)=\frac12e^{a+b}(1+e^{-2(a+b)})$ and $\cosh(b-a)=\frac12e^{b-a}(1+e^{-2(b-a)})$.
(E2) If $b\le u$ then $a\star b\le b\le u$. If $b>u$ then $a\star b\le u$ iff
$\tanh a\le\tanh u/\tanh b$. Finally $(1+\tanh u/\tanh b)/(1-\tanh u/\tanh b)=\sinh(b+u)/\sinh(b-u)$.
\end{proof}

\begin{proof}[Proof of Lemma~\ref{lem:markov}]
Let $\Omega$ denote the four pole spins of $\DL_N$ and $P_D$ the four spins at the poles of $D$.
Given $\Omega$ the replicas are independent, and conditioning further on $P_D$, an intersection of
one event for each replica, keeps them independent. In one replica the fixed-pole Gibbs measure is
a Markov random field, and by Lemma~\ref{lem:struct}(a) the set $\{a_D,b_D\}$ separates
$\operatorname{int}D$ from the rest of $\DL_N$. Hence given $(\Omega,P_D)$ each replica restricted
to $\operatorname{int}D$ has the fixed-pole law of $D$. This law does not depend on $\Omega$
(the poles $A,B$ are not interior to $D$), so it is also the conditional law given $P_D$ alone.
Bonds are independent given spins. Part (b) follows from Lemma~\ref{lem:rec}(d) and from the
argument of (a) with $P_D$ replaced by the spins at $a_D,y,b_D$, which separate
$\operatorname{int}H^{(1)}$, $\operatorname{int}H^{(2)}$ and the rest of $\DL_N$ by
Lemma~\ref{lem:struct}(a).
\end{proof}

\begin{proof}[Proof of Lemma~\ref{lem:rates}]
For $Y\ge0$ and a nonincreasing $C^1$ function $h$ with $h(\infty)=0$,
$\E h(Y)=\int_0^\infty(-h'(s))\Pp(Y<s)\,ds$; apply this with $h=\frac12\sech^2$, $\psi$,
$e^{-4s}$, $\sech^2$ and $1-\tanh$. Conditioning on the second half and using its independence
from the first (Lemma~\ref{lem:struct}(c)),
$\Pp(\Delta_k<s)\le\sup_y\Pp(||\xi^{(1)}_k|-y|<s)\le\min\{1,Q_{k-1}(s)\}$; and
$\Pp(\Xi_k<s)=\Pp(|\xi_{k-1}|<s)^2$. For the comparison, put $h_{\rm s}(s):=\frac12\sech^2s$.
Since $\psi(s)-h_{\rm s}(s)=p(s)^2$ is nonincreasing,
$-\psi'(s)\ge-h_{\rm s}'(s)$ for $s\ge0$.
Multiplying by $\min\{1,Q_{k-1}(s)\}$ and integrating proves
$\as_k\le\ab_k$, because the second integral defining $\ab_k$ is nonnegative.
\end{proof}

\begin{proof}[Proof of Lemma~\ref{lem:rogozin}]
For $m_i=\mathsf M(X_i)$ this is Rogozin's theorem (Section~\ref{sec:rogozin}; in
\cite[Thm.~1.4]{MMX2017} take $d=1$, $k=1$, $T=(1,\dots,1)$). It remains to show that $\Phi_g$ is
nonincreasing in each width. The sum $W$ of the other uniform variables is symmetric unimodal
by Wintner's theorem \cite{Wintner1938,DJ1988} (see Section~\ref{sec:su}), and the density at $0$ of $U+W$, with
$U$ uniform of width $v$ and independent of $W$, is $v^{-1}\Pp(|W|\le v/2)$; since
$t\mapsto\Pp(|W|\le t)$ is concave and nonnegative, $t\mapsto\Pp(|W|\le t)/t$ is nonincreasing.
\end{proof}

\begin{proof}[Proof of Lemma~\ref{lem:ratesH}]
Lemma~\ref{lem:Hcons}, and the integrals $\int_0^\infty\sech^2s\tanh s\,ds=\int_0^\infty s\sech^2s\tanh s\,ds=\frac12$,
$\psi(0)=\frac34$, $\int_0^\infty\psi=\frac12\ln2+\frac14$, $\int_0^\infty4e^{-4s}(w+2rs)^2\,ds=w^2+wr+\frac12r^2$,
$\int_0^\infty2t\sech^2t\tanh t\,dt=1$ and $\int_0^\infty t\sech^2t\,dt=\ln2$.
\end{proof}

\begin{proof}[Proof of Lemma~\ref{lem:U}]
$K_1\star K_2=e_1e_2(|K_1|\star|K_2|)$ with $e_1e_2$ a fair sign independent of the moduli, so it
suffices that $|K_1|\star|K_2|$ has a concave distribution function on $[0,\infty)$. By Khintchine
$|K_i|=U_i\Theta_i$ with $U_i$ uniform on $[0,1]$; mixtures preserve concavity, and
$\Theta_1\Theta_2=0$ gives $0$. So let $X,Y$ be independent and uniform on $[0,a]$ and $[0,b]$, $a,b>0$. For $z>0$ write
$s=\tanh z$, $t_a=\tanh a$, $t_b=\tanh b$. Differentiating $\Pp(X\star Y\le z)$ with (E2) of
Lemma~\ref{lem:series} and substituting $u=\tanh x$ gives the density
\[
\rho(z)=\frac1{ab}\int_{s/t_b}^{t_a}\frac{(1-s^2)u}{(u^2-s^2)(1-u^2)}\,du
=\frac1{2ab}\log\frac{(t_a^2-s^2)(t_b^2-s^2)}{(1-t_a^2)(1-t_b^2)s^2}
\qquad(s<t_at_b),
\]
and $\rho(z)=0$ for $s\ge t_at_b$, using
$(1-s^2)u/((u^2-s^2)(1-u^2))=u/(u^2-s^2)+u/(1-u^2)$. Each factor inside the logarithm decreases in
$s$, and the logarithm vanishes at $s=t_at_b$; so $\rho$ is nonincreasing on $(0,\infty)$ and there is
no atom. The statement for $\xi_k$ follows by induction from Wintner's theorem and the Gaussian base.
For $\minT$, the distribution function $1-(1-F)^2$ is concave when $F$ is concave, nondecreasing and
$[0,1]$-valued.
\end{proof}

\section{Certificates and their verification}\label{app:computation}

This appendix describes the computer-verified part of the paper. Appendix~\ref{sec:gauss-certs}
gives the certificates for Theorem~\ref{thm:A} (Tables~\ref{tab:thresholds} and~\ref{tab:oneshot})
and how they are checked. Appendices~\ref{app:decomp}--\ref{sec:checks} concern the five certificates of Theorem~\ref{thm:B} (Table~\ref{tab:points}): Appendix~\ref{app:decomp} proves the reduction of condition (B$j$) to pairwise sums, a step of the proof, and Appendices~\ref{sec:arith} and~\ref{sec:checks} give the arithmetic model and the individual checks. Appendix~\ref{sec:validation} reports that every certificate of both theorems passes and records validation runs, which are not part of the proof, and Appendix~\ref{sec:repro} gives the commands that reproduce every certified number. The code is at
\url{https://github.com/PeaBrane/mk-spin-glass-order}, commit
\href{https://github.com/PeaBrane/mk-spin-glass-order/tree/\repocommit}{\texttt{\repocommit}}.

\subsection{Gaussian certificates for \texorpdfstring{Theorem~\ref*{thm:A}}{Theorem A}}\label{sec:gauss-certs}

For standard Gaussian couplings, $\phi=1/\sqrt{2\pi}$. A certificate is found by a floating-point
search (per-level choice of $(L_k,R_k)$, and at every level an attempt to close the tail with
Lemma~\ref{lem:T}), frozen as exact rational input, and then verified in exact rational
arithmetic by the checker of the repository (pure Python, standard library only; entry point
\texttt{partI/check\_partI.py}, arithmetic in \texttt{partI/propa.py}): $\phi$, $e^{-2L}$, $e^{-2R}$ and $\ln2$ are enclosed by series with
directed rounding, every propagated quantity ($r_k$, $w_k$, $\varpi$ and the sums defining $\Gamma$)
is rounded upward, the constants $\kappa_g$ are exact rationals, and every hypothesis of
Definition~\ref{def:certA} is decided exactly. The imbalance claim is tested in the form
$(1-2\ell_{\rm b})_+^2-P>0$, which is equivalent to $2\ell_{\rm b}+\sqrt P<1$. The printed bounds
are rounded in the safe direction. Further implementations, used for validation, are described in
Appendix~\ref{sec:validation}.

\emph{One-shot certificates.} A certificate that closes at level $0$ ($K=0$) suffices for pole
order. With
$r_0=\phi/\beta$, $w=Ar_0^2$, $\rho=r_0/(1-Ar_0^2)$, $\varpi=2Ar_0^2+(2\rho R)^2+4\rho L$ and
$\vartheta=(1+e^{-2L})/\tanh R$, the hypotheses of Lemma~\ref{lem:T} at $K=0$ read
\begin{equation}\label{eq:oneshot}
\vartheta\sum_{g=1}^n\binom ng\kappa_g\,\varpi^{\,n-g}\le\gamma<1,\qquad \varpi^n\le A\gamma^2r_0^2,
\qquad Ar_0^2<1 .
\end{equation}
For $n=4$ the first is $\vartheta(\kappa_4+4\kappa_3\varpi+6\kappa_2\varpi^2+8\varpi^3)\le\gamma$. Each row of
Table~\ref{tab:oneshot} satisfies \eqref{eq:oneshot} in exact rational arithmetic with rigorous
one-sided bounds for $\phi$, $e^{-2L}$ and $\tanh R$; the check needs only a few lines of arithmetic.
These certificates give $\Hyp_k(r_0\gamma^k,Ar_0^2\gamma^{2k})$ and hence
Theorem~\ref{thm:A}(i) and the upper bound (v), but their bulk constants are not small enough for
positive bounds in (ii)--(iv) and (vi); the bulk thresholds of Table~\ref{tab:thresholds} use
multi-level certificates. The rows of Table~\ref{tab:oneshot} were checked in interval arithmetic
and re-verified in exact rationals in a separate audit by a reviewer who had not written the
checker.

\begin{table}[htb]
\centering
\begin{tabular}{rccccc}
\toprule
$n$ & $\beta$ & $L$ & $R$ & $A$ & $\gamma$\\
\midrule
4 & $525$ & $5/2$ & $21/4$ & $1/100$ & $499833/500000$ \\
5 & $135$ & $17/10$ & $4$ & $1/1000$ & $498759/500000$ \\
6 & $86$ & $7/5$ & $7/2$ & $1/10000$ & $997311/1000000$ \\
7 & $68$ & $6/5$ & $16/5$ & $1/10000$ & $199287/200000$ \\
8 & $59$ & $1$ & $3$ & $1/10000$ & $994437/1000000$ \\
12 & $44$ & $3/5$ & $13/5$ & $1/10000$ & $992579/1000000$ \\
16 & $36$ & $3/10$ & $7/3$ & $1/10000$ & $498581/500000$ \\
32 & $63/5$ & $1/1000$ & $5/4$ & $1/10000$ & $499357/500000$ \\
64 & $47/5$ & $1/1000$ & $5/6$ & $1/10000$ & $124603/125000$ \\
\bottomrule
\end{tabular}
\caption{One-shot pole certificates for standard Gaussian couplings, in exact rationals. With
$r_0=\phi/\beta$ each row satisfies the three inequalities \eqref{eq:oneshot}, hence the
hypotheses of Lemma~\ref{lem:T} at $K=0$. Consequently $\Hyp_k(r_0\gamma^k,Ar_0^2\gamma^{2k})$ holds for all
$k$, and Theorem~\ref{thm:A}\textup{(i)} and the upper bound \textup{(v)} hold for every larger
$\beta$; the bulk constants of these certificates give no positive bounds in
\textup{(ii)--(iv)} and \textup{(vi)}.}
\label{tab:oneshot}
\end{table}

\subsection{Reducing \texorpdfstring{(B$j$)}{(Bj)} to pairwise sums}\label{app:decomp}

\begin{lemma}[Decompositions]\label{lem:decomp}
Condition \textup{(B$j$)} holds if there are s.u.\ magnitude laws $\omega_1,\dots,\omega_n$ with
$\omega_1\lpk T_\star(\nu_j)$, a sequence of steps $\omega_k\lpk\omega_{i_1}\oplus\omega_{i_2}$ with
$k=i_1+i_2$, each of whose inputs is $\omega_1$ or an earlier output, and whose last output is
$\omega_n$ (for example sequential,
$\omega_{k+1}\lpk\omega_k\oplus\omega_1$, or binary doubling, $\omega_{2k}\lpk\omega_k\oplus\omega_k$), and
$\nu_{j+1}\lpk\omega_n$. Any s.u.\ magnitude law $\nu'$ with $F_{\nu'}\ge F_\nu$ may replace a law $\nu$
in such a chain.
\end{lemma}

\begin{proof}
By induction and Corollary~\ref{cor:monotone},
$\omega_k\lpk\omega_1^{\oplus k}\lpk T_\star(\nu_j)^{\oplus k}$; take $k=n$ and use transitivity.
\end{proof}

A coarsening of a law onto a coarser grid is allowed in this chain only if its distribution
function lies above the original one and is again concave; interpolating a concave distribution
function at coarser knots produces a chord below it, which is the wrong direction. Check (V1) of
Appendix~\ref{sec:checks} verifies each coarsening as a comparison of this kind.

\subsection{Arithmetic model}\label{sec:arith}

The certificates of Theorem~\ref{thm:B} are checked in the following arithmetic model.

\emph{Laws.} A law is a staircase law: an atom at $0$ and masses on the bins of a uniform grid, stored
as integers in units of $2^{-48}$, with nonincreasing bin masses that, together with the atom, sum
exactly to $1$, that is, to $2^{48}$ units. Its distribution function is concave and piecewise
linear, which is (S). The shape $X_c$ is such a law with no atom and support in $[0,x_{\max}]$; its
distribution function $F_c$ is evaluated in units of its grid.

\emph{Exact parts.} The $n$-fold sums, $\minT$, coarsenings, the output of $T_\star$,
(C1)--(C2), and (EN) are decided with integers and exact rationals. Polynomial products use Kronecker
substitution, validated against the naive product.

\emph{Transcendental parts.} The values of $\delta$, $\operatorname{erf}$, $\psi$, $e^{-4s}$ and $\ln2$ are
enclosed by interval arithmetic on 64-digit decimals with rounding toward $-\infty$ and $+\infty$;
$\exp$, $\log$, $\arctan$ and $\operatorname{erf}$ are Taylor series with explicit remainder bounds,
$\pi$ is computed by Machin's formula, and square roots are verified by squaring. Integers are taken as
the ceiling of an upper endpoint or the floor of a lower endpoint, in the direction that keeps an upper
bound of a distribution function an upper bound. Floating-point numbers are used in the search, and in
the check only to make choices (tangent points, argument reductions, the $\operatorname{erf}$ branch
and working precision) whose validity is re-verified exactly or does not depend on the choice; no
floating-point comparison decides an inequality or the termination of a series. The checker accepts
only integer values where integers are expected, and records the SHA-256 hashes of all inputs.

\subsection{The checks}\label{sec:checks}

On the grid of mesh $h$, bin $k$ is $[kh,(k+1)h]$, and a law has knot values (the atom plus the
first $k$ masses) and bin masses. Checks (F1)--(F4) verify the family and (V0)--(V4) the bridge and
the entry condition. Every law is validated as a staircase law when it is read or produced, which gives
(S) for the bridge laws and the first part of (F1) for the family; (V0) verifies (B0), (V1)--(V3) give
(B$j$) through Lemma~\ref{lem:decomp}, and (V4) verifies (EN); (K) evaluates the constants.
\begin{enumerate}
\item[(F1)] \emph{Validity.} Every law is a staircase law; $X_c$ is atomless and supported in the first
bins up to $x_{\max}$.
\item[(F2)] \emph{$\minT$.} $G_1\ge$ the distribution function of $\min(X_1,X_2)$, $X_i$ iid $X_c$. On
each bin the slack is a quadratic polynomial in the position within the bin, and nonnegativity is
decided exactly at both ends and at the vertex.
\item[(F3)] \emph{Sums.} $G_m\ge$ the distribution function of $\eta_{m-1}\oplus\eta_1$ for
$m=2,\dots,n$; the exact distribution function of $|X+Y|$ for independent signed versions $X,Y$ of two
staircase laws is piecewise quadratic, and the same per-bin quadratic test applies. Together with (F2)
this is (C0).
\item[(F4)] \emph{Family conditions.} (C1) and (C2) are evaluated exactly at both ends of every
elementary interval between consecutive knots of $F_c(\cdot/\lambda)$ and of all $G_m$, each interval
with its own slopes; the checker uses $\max_{m'\ge m}\sH_{m'}\ge\hsH_m$, a stronger condition. On each
elementary interval the margin of (C2) is concave (affine minus a positive combination of maxima of
affine functions), so the endpoints suffice. The parameter conditions $\lambda>1$,
$A,\bareps>0$ and $2A\bareps\le1$ are checked as well.
\item[(V0)] \emph{Level $0$.} On each bin where $F_0<1$, the function $\operatorname{erf}(\cdot/(\beta\sqrt2))-F_0$
is concave; its tangent line at a point of the bin is enclosed in intervals and must be nonpositive at
both ends.
\item[(V1)] \emph{Coarsening.} When the grid mesh doubles, the new distribution function must be at least
the old one; both are linear on each old bin, so comparing knot values suffices.
\item[(V2)] \emph{Series map $T_\star$.} On a fine grid $u_i$, integers $U_i$ with
$U_i2^{-48}\ge\Pp(Y_1\star Y_2\le u_i)$ are computed from (E2) of Lemma~\ref{lem:series}: $Y_2$ is
uniform within each bin, $\iota(u,b)$ is
decreasing in $b$, and $\iota$ is enclosed through the identity $\iota=u+\delta(b-u)-\delta(b+u)$ with
upward-rounded tables for $\delta(b-u)$, downward-rounded tables for $\delta(b+u)$, and a tail term using
$\delta(b+u)>0$. The distribution function $F$ of the output law must satisfy
$F(u_i)\ge U_{i+1}2^{-48}$; then $F(t)\ge\Pp(Y_1\star Y_2\le t)$ on $[u_i,u_{i+1}]$ by monotonicity.
\item[(V3)] \emph{$n$-fold sum.} The sum over $n$ branches is built by a fixed plan of pairwise sums
(for example $1+1\to2$, $2+2\to4$, $4+4\to8$ for $n=8$), each checked as in (F3); by
Lemma~\ref{lem:decomp} this gives (B$j$).
\item[(V4)] \emph{Entry.} (EN) with $L_0=c_0^+/\bareps$ and $\alpha_0'=Ac_0^-/L_0$, where $c_0^\pm$ are
rationals with $c_0^-<c_0<c_0^+$ verified by the interval enclosure of $\ln2$; the check is exact at the
knots of both piecewise-linear sides.
\item[(K)] \emph{Constants.} For $l-1\le J$ the bound $\bar\phi_l$ is evaluated bin by bin as an exact
rational, after replacing each transcendental factor by the ceiling of its upper endpoint; for
$l-1>J$ the closed form of Section~\ref{sec:proofB} is used with the rational upper bound $c_0^+$ and a
rational upper bound for $\frac12\ln2+\frac14$. The sums \eqref{eq:kappas} are evaluated in exact
rationals, truncated (dropping nonnegative terms), and printed as lower bounds; the base-$10$
logarithms in Table~\ref{tab:points} are rounded down.
\end{enumerate}

\emph{Trusted base.} Soundness rests on the modules that evaluate the inequalities and the enclosures,
and on the drivers that sequence them. The plan of pairwise sums in (V3) is itself verified before
use: each step adds two labels that are already available, the new label is their sum, and the last
label is $n$. The constructors of candidate laws are not trusted: their outputs are always
re-verified, so an error there can only cause a rejection.

\subsection{Results and validation}\label{sec:validation}

All five certificates of Theorem~\ref{thm:B} and all certificates behind
Tables~\ref{tab:thresholds} and~\ref{tab:oneshot} pass. The whole verification uses only the Python
standard library, and the five certificates of Theorem~\ref{thm:B} take minutes each on one core.
The directories \texttt{partI/} and \texttt{partII/} of the repository contain the checks of the certificates of Tables~\ref{tab:thresholds} and~\ref{tab:oneshot} (Theorem~\ref{thm:A}) and of Table~\ref{tab:points} (Theorem~\ref{thm:B}). The following checks are not part of the proof.
\begin{itemize}
\item Before the inputs were frozen, the multi-level certificates of the original search for
Theorem~\ref{thm:A} were accepted by two earlier implementations, one in Arb ball arithmetic
\cite{Arb} and one in mpmath interval arithmetic \cite{mpmath}, sharing no code with each other;
the second is included in the repository, and the published checker is a port of them. Two
further implementations, written independently of all three from the statements of
Proposition~\ref{prop:A}, Lemma~\ref{lem:T} and Section~\ref{sec:transfer} alone (one in exact
rationals, one in ball arithmetic), accept every frozen certificate; their code and outputs are in
\texttt{validation/} of the repository.
\item Three implementations of the branch sums, the most expensive step, produce the same laws
bit for bit; a hash of the whole chain of laws agrees across machines and across two
implementations of the transcendental tables.
\item The repository runs $46$ negative controls, each changing one thing in a frozen input, and all
are rejected. For the certificates of Theorem~\ref{thm:B} ($n=8$ unless stated): lowering $J$ by one; lowering $\beta$ by $0.1$ with
the frozen schedule (from $T^*=2$ to $T=5/2$); $\lambda=2$; $\bareps$ multiplied by $10$; a
violation of (C1) by one unit at $t=0$ ($n=4$); a wrong family hash; a wrong $n$, on the command
line and in the file; a JSON float for $J$, and for a mass ($n=4$); a per-law hash mismatch
($n=4$); two tampered plans of pairwise sums; and a tampered manifest hash. For the certificates of Theorem~\ref{thm:A}: $\gamma$ lowered by $10^{-5}$ or
multiplied by $0.999$; $A$ divided by $1000$; every certificate of Table~\ref{tab:thresholds} at
$\beta-0.1$ with the same parameters; a dropped last level; a wrong $n$; a one-shot certificate of
Table~\ref{tab:oneshot} claiming order; $\beta$ written as a JSON float; and a tampered manifest
hash.
\item In a separate audit of the checker by a reviewer who had not written it, the verification
was repeated from the frozen inputs, on a different machine and with the standard library only;
it reproduced $J$, the margins, the hashes of the chains of laws and the constants. The same audit
built inputs that violate (C1), (C2), (C0), (V3), (V0) and (EN) by a small amount, each of which
passes a weaker test, and the checker rejected all of them. This audit and rerun used an earlier
version of the checker; the published version
adds the checks described in Appendices~\ref{sec:arith}--\ref{sec:checks} (verified plan labels, an
exact stopping rule for the $\operatorname{erf}$ series, strict input types, an explicit test of
$2A\bareps\le1$) and reproduces the same laws, margins and constants.
\item A second implementation, written earlier with its own search and its own exact checker, certifies
$n=3,4,5,8$ at the same temperatures, and $n=7$ once its checker is extended to verify the
intermediate sums of the plan $1+1\to2$, $2+2\to4$, $4+2\to6$, $6+1\to7$ (the unextended checker
accepts only plans whose intermediate sums are powers of two or $n$). Its code, the extension and the
records of both runs are in \texttt{validation/} of the repository.
\end{itemize}
The construction of $X_c$ (a scaled floating-point approximation of a zero-temperature fixed shape),
the choice of grids, schedules and parameters, and the population-dynamics estimates of $T_c(n)$ are
heuristic and play no role in the proof.

\subsection{Reproduction}\label{sec:repro}

The repository contains the checker, the frozen inputs with their hashes, the scripts that produce
Tables~\ref{tab:thresholds}, \ref{tab:points} and~\ref{tab:oneshot} from result files, and the certificates of
Appendix~\ref{sec:gauss-certs}. The command
\begin{verbatim}
sh run_all.sh
\end{verbatim}
verifies every certified number in this paper, runs the negative controls for both theorems, regenerates
the three table bodies and compares them with the committed ones; it needs Python 3.9 or later and no
third-party package, and takes about 7 minutes on one core. Single steps can be run separately, for
example the certificates of Theorem~\ref{thm:A} by
\begin{verbatim}
python3 partI/check_partI.py --manifest partI/inputs/SHA256SUMS.txt --out-dir out
\end{verbatim}
and one row of Table~\ref{tab:points} by
\begin{verbatim}
python3 partII/check.py --n 4 --point partII/inputs/point_n4.json \
        --kernel python --threads 1 --out check_n4.json
\end{verbatim}

The code and \texttt{partII/CONDITIONS.md} use the labels of an earlier write-up: its
Theorems~1A, 1B and~1 are Theorems~\ref{thm:family}, \ref{thm:bridge} and~\ref{thm:growth} here
(not Theorems~\ref{thm:A} and~\ref{thm:B}), its Theorems~2 and~3 are items (ii)--(iv) of
Theorem~\ref{thm:B}, its condition (E) is (EN), and its conditions (B$j$.1) and
(B$j$.2)--(B$j$.3) are checked by (V2) and (V3) of Appendix~\ref{sec:checks}, respectively.


\end{document}